\documentclass[a4paper, 11pt]{article}
\usepackage[utf8]{inputenc}
\usepackage[T1]{fontenc}

\usepackage[
	sorting=nyt,
	backend = biber,
	hyperref = auto,
	backref = true,
	doi = false,
	url = false
]{biblatex}

\usepackage{amsthm}
\usepackage{amssymb}
\usepackage{mathtools}
\usepackage{stmaryrd}	\SetSymbolFont{stmry}{bold}{U}{stmry}{m}{n}	
\usepackage{bm}
\usepackage{nccrules}
\usepackage{euscript}
\usepackage{mathrsfs}
\usepackage{euscript}
\usepackage{upgreek}
\usepackage{tikz-cd}
\usetikzlibrary{positioning, shapes.geometric, patterns, graphs, decorations.pathmorphing}

\usepackage{paralist}

\usepackage[colorlinks]{hyperref}
\hypersetup{
	linkcolor=blue,
	citecolor=red,
	urlcolor=teal,
	pdftitle = {Spectral transfer for metaplectic groups. II. Hecke algebra correspondences},
	pdfauthor = {Fei Chen, Wen-Wei Li}
}

\newcommand{\Z}{\ensuremath{\mathbb{Z}}}

\newcommand{\R}{\ensuremath{\mathbb{R}}}
\newcommand{\CC}{\ensuremath{\mathbb{C}}}

\newcommand{\Tr}{\operatorname{tr}}

\newcommand{\Gal}{\operatorname{Gal}}
\newcommand{\Frob}{\operatorname{Frob}}
\newcommand{\Weil}[1]{\ensuremath{\mathrm{Weil}_{#1}}}	

\newcommand{\dd}{\mathop{}\!\mathrm{d}}
\newcommand{\mes}{\operatorname{mes}}	

\newcommand{\lrangle}[1]{\ensuremath{\langle #1 \rangle}}

\newcommand{\identity}{\ensuremath{\mathrm{id}}}

\newcommand{\Hom}{\operatorname{Hom}}

\newcommand{\End}{\operatorname{End}}
\newcommand{\rightiso}{\ensuremath{\stackrel{\sim}{\rightarrow}}}

\newcommand{\dcate}[1]{\ensuremath{\text{-}\mathsf{#1}}}	

\newcommand{\Ker}{\operatorname{ker}}

\newcommand{\Image}{\operatorname{im}}

\newcommand{\Gm}{\ensuremath{\mathbb{G}_\mathrm{m}}}
\newcommand{\Ga}{\ensuremath{\mathbb{G}_\mathrm{a}}}

\newcommand{\GL}{\operatorname{GL}}

\newcommand{\SO}{\operatorname{SO}}
\newcommand{\Or}{\operatorname{O}}

\newcommand{\SL}{\operatorname{SL}}
\newcommand{\Sp}{\operatorname{Sp}}

\newcommand{\Mp}{\ensuremath{\widetilde{\mathrm{Sp}}}}
\newcommand{\MMp}{\operatorname{Mp}}

\newcommand{\bmu}{\ensuremath{\bm\mu}}
\newcommand{\bpsi}{{\ensuremath{\uppsi}}}
\newcommand{\Endo}{\ensuremath{\mathcal{E}}}
\newcommand{\orbI}{\ensuremath{\mathcal{I}}}

\newcommand{\elli}{\operatorname{ell}}
\newcommand{\asp}{\ensuremath{\dashrule[.7ex]{2 2 2 2}{.4}}} 
\newcommand{\rev}{\ensuremath{\mathbf{p}}} 
\newcommand{\Trans}{\ensuremath{\mathcal{T}}}	
\newcommand{\trans}{\ensuremath{\check{\mathcal{T}}}}	

\theoremstyle{plain}
\newtheorem{proposition}{Proposition}
\newtheorem{lemma}[proposition]{Lemma}
\newtheorem{theorem}[proposition]{Theorem}

\theoremstyle{definition}
\newtheorem{definition}[proposition]{Definition}
\newtheorem{definition-theorem}[proposition]{Definition--Theorem}
\newtheorem{definition-proposition}[proposition]{Definition--Proposition}
\newtheorem{remark}[proposition]{Remark}

\numberwithin{equation}{section}
\numberwithin{proposition}{section}

\usepackage{amsmath}
\usepackage[nottoc]{tocbibind}

\usepackage{geometry}
\title{Spectral transfer for metaplectic groups. II. Hecke algebra correspondences}
\author{Fei Chen \quad Wen-Wei Li}
\date{}

\DeclareFieldFormat{postnote}{#1}	
\makeatletter
\renewcommand{\l@section}{\@dottedtocline{1}{1.5em}{2.0em}}
\renewcommand{\l@subsection}{\@dottedtocline{2}{4.0em}{3.0em}}
\makeatother

\begin{document}
	
\maketitle

\begin{abstract}
	Let $\mathrm{Mp}(2n)$ be the metaplectic group over a local field $F \supset \mathbb{Q}_p$ defined by an additive character of $F$ of conductor $4\mathfrak{o}_F$. Gan--Savin ($p \neq 2$) and Takeda--Wood ($p=2$) obtained an equivalence between the Bernstein block of $\mathrm{Mp}(2n)$ containing the even (resp.\ odd) Weil representation and the Iwahori-spherical block of the split $\mathrm{SO}(2n+1)$ (resp.\ its non-split inner form), by giving an isomorphism between Hecke algebras. We revisit this equivalence from an endoscopic perspective. It turns out that the L-parameters of irreducible representations are preserved, whilst the difference between characters of component groups is governed by symplectic local root numbers.
\end{abstract}
	
{\scriptsize
	\begin{tabular}{ll}
		\textbf{MSC (2020)} & Primary 22E50; Secondary 11F70, 20C08 \\
		\textbf{Keywords} & Hecke algebra, metaplectic group, local Langlands correspondence
	\end{tabular}}
	
\setcounter{tocdepth}{1}
\tableofcontents
	
\section{Introduction}\label{sec:intro}
\subsection{Overview}
Let $F$ be a non-Archimedean local field of characteristic zero. For a symplectic $F$-vector space $(W, \lrangle{\cdot|\cdot})$ of dimension $2n$, where $\lrangle{\cdot|\cdot}: W \times W \to F$ is the symplectic form, denote by $G$ the symplectic group $\Sp(W)$, often written as $\Sp(2n)$. We have the twofold metaplectic covering $\tilde{G}^{(2)} := \MMp(W)$ of the locally compact group $G(F)$, often written as $\MMp(2n)$. A smooth representation of $\tilde{G}^{(2)}$ is said to be genuine if the kernel $\bmu_2 := \{\pm 1\} \subset \CC^{\times}$ of $\tilde{G}^{(2)} \twoheadrightarrow G(F)$ acts tautologically.

To every non-trivial additive character $\bpsi: F \to \CC^{\times}$ is attached the Weil representation $\omega_{\bpsi}$ of $\tilde{G}^{(2)}$. It is genuine and decomposes into $\omega_{\bpsi}^+ \oplus \omega_{\bpsi}^-$ where $\omega_{\bpsi}^+$ (resp.\ $\omega_{\bpsi}^-$) is irreducible, called the even (resp.\ odd) part of $\omega_{\bpsi}$. In what follows, we assume that $\bpsi|_{4\mathfrak{o}} = \mathbf{1}$ and $\bpsi|_{4\varpi^{-1}\mathfrak{o}} \not\equiv \mathbf{1}$, where $\mathfrak{o} = \mathfrak{o}_F \subset F$ is the ring of integers and $\varpi \in \mathfrak{o}$ is any uniformizer.

It is known that $\omega_{\bpsi}^{\pm}$ belong to different Bernstein blocks $\mathcal{G}_{\bpsi}^{\pm}$ of the category of genuine smooth representations of $\tilde{G}^{(2)}$. On the other hand, let
\[ G^+ := \text{the split}\; \SO(2n+1), \quad G^- := \text{its non-split inner form over $F$.} \]
Let $\mathcal{G}^{\pm}$ be the Iwahori-spherical Bernstein block of the category of smooth representations of $G^{\pm}(F)$, respectively. Most of the following statements will consist of $+$ and $-$ versions.

Denote by $p$ the residual characteristic of $F$. Gan and Savin \cite{GS2} established equivalences of categories $\mathcal{G}_{\bpsi}^{\pm} \simeq \mathcal{G}^{\pm}$ when $p \neq 2$, which is then extended to all $p$ by Takeda and Wood \cite{TW18}. This is achieved by constructing types for $\mathcal{G}_{\bpsi}^{\pm}$. Specifically, they obtained:
\begin{itemize}
	\item $\CC$-algebras $H_{\bpsi}^{\pm}$ with explicit presentations;
	\item equivalences of categories $\mathcal{G}_{\bpsi}^{\pm} \simeq H_{\bpsi}^{\pm}\dcate{Mod}$ using type theory;
	\item explicit isomorphisms of algebras $\mathrm{TW}: H^{\pm} \rightiso H_{\bpsi}^{\pm}$ where $H^{\pm}$ is the Iwahori--Hecke algebra for $G^{\pm}$;
	\item the resulting equivalence $\mathrm{TW}^*: \mathcal{G}_{\bpsi}^{\pm} \rightiso \mathcal{G}^{\pm}$ induced from $\mathrm{TW}$.
\end{itemize}

In \cite{TW18}, $\mathrm{TW}$ is shown to be an isomorphism of Hilbert algebras, a notion introduced in \cite[\S 3.1]{BHK11}. Hence temperedness and square-integrability of representations are preserved under $\mathrm{TW}^*$.

On the other hand, Gan and Savin \cite{GS1} proved the local Langlands correspondence for $\tilde{G}^{(2)}$, abbreviated as LLC, by leveraging the $\Theta$-lifting for the dual pairs
\[ (\Sp(W), \Or(V^{\pm})), \quad \dim V^{\pm} = 2n+1, \; \text{discriminant} = 1, \quad \text{Hasse invariant} = \pm 1. \]
This is feasible because the LLC for $G^+ = \SO(V^+)$ and $G^- = \SO(V^-)$ are furnished by the works of Arthur \cite{Ar13} and Ishimoto \cite{Is24}, respectively; both are based on endoscopy. All these groups share the same Langlands dual group $\Sp(2n, \CC)$ equipped with trivial Galois action. Therefore, to each irreducible object $\pi$ (resp.\ $\sigma$) of $\mathcal{G}_{\bpsi}^{\pm}$ (resp.\ $\mathcal{G}^{\pm}$) is attached an enhanced parameter $(\phi, \chi)$ (resp.\ $(\phi^\circ, \chi^\circ)$), expressed as
\[ \pi = \pi_{\phi, \chi}, \quad \sigma = \sigma_{\phi^{\circ}, \chi^{\circ}}. \]
Here $\chi$ is a character of the component group $\EuScript{S}_\phi$ of the centralizer $S_\phi$ of the L-parameter $\phi$ in $\Sp(2n, \CC)$. Ditto for $\chi^\circ$.

\textbf{Question 1}. If $\sigma = \mathrm{TW}^*(\pi)$, how are $(\phi, \chi)$ and $(\phi^{\circ}, \chi^{\circ})$ related?

There is no apparent link between $\Theta$-lifting and $\mathrm{TW}^*$. Nonetheless, Gan and Savin \cite[Corollary 21]{GS2} showed that $\mathrm{TW}^*$ preserves L-parameters: $\phi = \phi^{\circ}$, whereas $\chi$ and $\chi^{\circ}$ differ in general. The following issue remains unsettled in \textit{loc.\ cit.}

\textbf{Question 2.} In this scenario, what is the exact difference between $\chi$ and $\chi^\circ$?

The aim of this article is to address both questions from an endoscopic perspective. We reprove the equality $\phi = \phi^{\circ}$ and show that the difference between $\chi$ and $\chi^\circ$ is governed by certain symplectic local root numbers, stated as follows.

\begin{definition}[= Definition \ref{def:nu-phi}]
	Let $\phi$ be an L-parameter $\mathcal{L}_F \to \Sp(2n, \CC)$, where $\mathcal{L}_F := \Weil{F} \times \SL(2, \CC)$. Decompose $\phi$ into a direct sum of simple representations of $\mathcal{L}_F$. There is a canonical isomorphism $\EuScript{S}_\phi \simeq \bmu_2^{I^+}$, where $I^+$ is the indexing set of self-dual simple summands in $\phi$ of symplectic type. Identify the Pontryagin dual $\EuScript{S}_\phi^\vee$ of $\EuScript{S}_\phi$ with $\bmu_2^{I^+}$, and define $\nu_\phi \in \EuScript{S}_\phi^\vee$ by
	\[ \nu_{\phi, i} := \epsilon\left(\frac{1}{2}, \phi_i, \bpsi \right), \quad i \in I^+. \]
	where $\phi_i: \mathcal{L}_F \to \GL(m_i, \CC)$ is the corresponding simple summand in $\phi$, and $\epsilon(\frac{1}{2}, \cdots)$ is the local root number defined in \cite[\S 2.2]{GR10}.
\end{definition}

For all $p$, the character $\nu_\phi$ of $\EuScript{S}_\phi$ is generally non-trivial.

\begin{theorem}[= Theorem \ref{prop:main}]
	\label{prop:main-preview}
	Let $\pi = \pi_{\phi, \chi}$ be an irreducible object of $\mathcal{G}_{\bpsi}^{\pm}$ and $\sigma = \sigma_{\phi^{\circ}, \chi^{\circ}} := \mathrm{TW}^*(\pi)$. Then
	\[ \phi^\circ = \phi, \quad \chi^\circ = \chi\nu_\phi. \]
\end{theorem}

The LLC for the Iwahori-spherical block $G^{\pm}$ is expected to agree with Lusztig's parametrization in \cite[Corollary 6.5]{Lu95}; in particular, $\phi$ should be trivial on the inertia subgroup $I_F$. Since we allow general residual characteristic $p$, the compatibility between these parametrizations is not known yet. Nonetheless, the following (very weak) result is a by-product of our approach.

\begin{proposition}[= Proposition \ref{prop:phi-nr}]
	In the scenario above, $\phi$ is trivial on $I_F$.
\end{proposition}

Besides tying up a loose end in \cite{GS2, TW18}, we also expect Theorem \ref{prop:main-preview} to play a role in studying the metaplectic local intertwining relation (LIR) beyond the tempered case, via a global-to-local approach.

\subsection{Idea of the proof}
Below is an overview of our approach. In the basic case where $n=1$ and $\pi$ is square-integrable, everything can be checked by hand; see \cite{GanP} for a catalog. The general case will be reduced to the basic one by induction on $n$.

First of all, we need to know that $\mathrm{TW}^*: \mathcal{G}_{\bpsi}^{\pm} \rightiso \mathcal{G}^{\pm}$ is compatible with Jacquet modules and parabolic inductions, both are assumed to be normalized. This is not included in the standard machinery of type theory except in the $+$ case with $p > 2$, but our prior work \cite{CL23} provides the required properties, as summarized in Theorem \ref{prop:compatibility-1}. See also Remark \ref{rem:splitting-Levi}.

Based on this compatibility, the reduction to the basic case proceeds as follows. 
\begin{description}
	\item[Reduction to tempered case] The LLC in \cite{GS2} commutes with the formation of Langlands quotients, and so does $\mathrm{TW}^*$.
	
	\item[Reduction to good parity] We say a bounded L-parameter $\phi: \mathcal{L}_F \to \Sp(2n, \CC)$ is of good parity if all simple summands in $\phi$ are self-dual of symplectic type. Using the compatibility between LLC (or $\mathrm{TW}^*$) and parabolic induction, the problem reduces easily to the case where the L-parameter of $\phi$ is of good parity.
	
	\item[Reduction to square-integrable case] Tempered representations are obtained from square-integrable ones by parabolic induction. The precise classification is known as Knapp--Stein theory, whose relation to LLC is clarified by Arthur's local intertwining relation \cite[Theorem 2.4.1]{Ar13}, abbreviated as LIR. We analyze this procedure through the tempered LIR for $\tilde{G}^{(2)}$ due to Ishimoto \cite{Is20}. To compare it with the $G^{\pm}$-side, which is also due to Ishimoto \cite{Is24} in the $-$ case, we need the compatibility of $\mathrm{TW}^*$ with normalized intertwining operators from \cite{CL23}, as summarized in Theorem \ref{prop:compatibility-2}.
	
	Note that the Knapp--Stein theory has been generalized to covering groups by C.\ Luo \cite{Luo20b}.
	
	\item[Reduction via Jacquet modules] Assume $n > 1$. Suppose $\pi$ is square-integrable, hence so is $\sigma$. We compare appropriate partial Jacquet modules of $\pi$ and $\sigma$ via $\mathrm{TW}^*$ in order to reduce the problem to metaplectic groups of smaller rank. The arguments for $G^{\pm}$ are due to Moeglin, see \cite[\S 8.3.4]{Rel18}. They are adapted to the metaplectic side by using the endoscopic character relations (abbreviation: ECR) of C.\ Luo \cite{Luo20}, as summarized in Theorem \ref{prop:Luo-endo}, together with auxiliary results from \cite{C24}.
\end{description}

We remark that the metaplectic ECR was first conceived in \cite{Li19}, a prequel of the present work. It involves the set $\Endo_{\elli}(\tilde{G})$ of elliptic endoscopic data for $\tilde{G}$, as well as the spectral transfer $\trans_{\mathbf{G}^!, \tilde{G}}$ of distributions from the endoscopic group $G^!(F)$ to $\tilde{G}$ for each $\mathbf{G}^! \in \Endo_{\elli}(\tilde{G})$. Compared with Arthur's ECR, it features an extra local root number; see \S\ref{sec:ECR}.

\begin{remark}
	Although this article is about Hecke algebra correspondences, our arguments make no direct use of type theory or Hecke algebras; all such ingredients are encapsulated in the compatibility properties of $\mathrm{TW}^*$ that we import from \cite{CL23}.
\end{remark}

\begin{remark}\label{rem:splitting-Levi}
	Levi subgroups of $G$ can be expressed as
	\[ M := \Sp(W^\flat) \times \prod_{k=1}^r \GL(n_k), \]
	where $W^{\flat}$ is a symplectic subspace of $W$ with $\dim W^{\flat} + 2\sum_{k=1}^r n_k = 2n$. In order to compare $\mathrm{TW}^*$ with its avatar for the preimage of $M(F)$, we will actually work with the eightfold covering $\tilde{G} = \Mp(W)$ of $G(F)$, defined as the push-out of the twofold covering $\tilde{G}^{(2)} = \MMp(W)$ along $\bmu_2 \hookrightarrow \bmu_8 := \{z \in \CC^{\times}: z^8 = 1\}$. The category $\tilde{G}\dcate{Mod}$ of genuine smooth representation of $\tilde{G}$ is defined in the evident manner, and it is equivalent to that of $\tilde{G}^{(2)}$.
	
	Using the chosen $\bpsi$ and $\lrangle{\cdot|\cdot}$, the preimage $\tilde{M}$ of $M(F)$ in $\tilde{G}$ splits canonically as
	\[ \tilde{M} = \Mp(W^\flat) \times \prod_{k=1}^r \GL(n_k, F) \]
	with the convention that $\Mp(W^{\flat}) = \bmu_8$ if $W^{\flat} = \{0\}$. The avatars of $\mathcal{G}_{\bpsi}^{\pm}$, $H_{\bpsi}^{\pm}$ and $\mathrm{TW}$ can then be defined in the evident way; see \eqref{eqn:H-decomp} and \eqref{eqn:TW-M}. This is the convention adopted in \cite{Li19, CL23, C24}, and will be reviewed in \S\S\ref{sec:Mp}--\ref{sec:splitting-Levi}.
\end{remark}

\subsection{Structure of this article}
In \S\ref{sec:review} we summarize various definitions and conventions about metaplectic groups, the Hecke algebras $H_{\bpsi}^{\pm}$, $H^{\pm}$, the Takeda--Wood isomorphism $\mathrm{TW}$ together with its compatibility with Jacquet modules and parabolic inductions. The crucial invariant $t(w)$ for $w \in \Omega^G_0$ will also be reviewed there, where $\Omega^G_0$ denotes the Weyl group of $G$.

In \S\ref{sec:statement} we recall the LLC (Gan--Savin), the theory of endoscopy for $\tilde{G}$, and ECR (Luo) for metaplectic groups, in order to state the main results. We also prove the equality between L-parameters in the tempered case.

The basic case $n=1$ and $\pi$ square-integrable is settled in \S\ref{sec:basic}. We then reduce the main theorem to the tempered case of good parity in \S\ref{sec:reduction-to-tempered}, and then to square-integrable $\pi$ in \S\ref{sec:reduction-to-L2} through metaplectic LIR (Ishimoto).

Finally, the proofs of Theorem \ref{prop:main} and Proposition \ref{prop:phi-nr} will be completed in \S\ref{sec:Jacquet} by using Jacquet modules to reduce to metaplectic groups of smaller rank.

\subsection{Acknowledgement}
The second named author is supported by NSFC, Grant No.\ 11922101 and 12321001.

\subsection{Conventions}
Throughout this article, $F$ is a chosen non-Archimedean local field of characteristic zero. The normalized absolute value on $F$ is denoted by $|\cdot| = |\cdot|_F$. Denote by $\mathfrak{o}$ the ring of integers in $F$, and let $\varpi \in \mathfrak{o}$ be any uniformizer.

Denote the Weil group of $F$ as $\Weil{F}$, and let $I_F$ be its inertia subgroup. The local Langlands group (i.e.\ the Weil--Deligne group) of $F$ is $\mathcal{L}_F := \Weil{F} \times \SL(2, \CC)$.

By an additive character of $F$, we mean a non-trivial continuous homomorphism $\bpsi: F \to \CC^{\times}$. For all $a \in F^{\times}$, set $\bpsi_a(x) = \bpsi(ax)$ so that $\bpsi_a$ is still an additive character of $F$. Further assumptions will be imposed on $\bpsi$ in \S\ref{sec:Mp}.

All representations in this article are over $\CC$. Trivial representations are denoted as $\mathbf{1}$. The central character of a representation $\tau$, if exists, is denoted by $\omega_\tau$.

For every $m \in \Z_{\geq 1}$, put $\bmu_m := \{z \in \CC^{\times}: z^m = 1 \}$. For every ring $A$, denote by $A\dcate{Mod}$ the abelian category of left $A$-modules. Denote by $\mathfrak{S}_n$ the permutation group on $n$ letters.

The character lattice of an $F$-torus $T$ is denoted by $X^*(T)$.

For each linear algebraic $F$-group $H$, the locally compact group of its $F$-points is denoted by $H(F)$. Hereafter, assume $H(F)$ is unimodular. Denote by $\mes(H)$ the $\CC$-line spanned by Haar measures of $H(F)$. We denote by $H\dcate{Mod}$ the abelian category of smooth representations of $H(F)$. The adjective ``smooth'' will often be omitted in this article.

Let $C^\infty_c(H(F))$ be the space of locally constant functions $f: H(F) \to \CC$ of compact support. The trace (or character) of an admissible smooth representation $\pi$ of $H(F)$ is denoted by $\Tr(\pi)$, which is a linear map
\[\begin{tikzcd}[row sep=tiny]
	C^\infty_c(H(F)) \otimes \mes(H) \arrow[r] & \CC \\
	f \otimes \mu \arrow[mapsto, r] & \int_{H(F)} f(h) \pi(h) \dd\mu(h),
\end{tikzcd}\]
where $\mu$ is a Haar measure on $H(F)$. Linear maps $C^\infty_c(H(F)) \otimes \mes(H) \to \CC$ are also called distributions on $H(F)$.

When $H$ is connected reductive, the parabolic inductions (resp.\ Jacquet functors) are always normalized and denoted by $i_P$ (resp.\ $r_P$), where $P \subset H$ is the parabolic subgroup in question.

The conventions above about representations also pertain to finite coverings of $H(F)$; see \S\ref{sec:Mp} for details.

Assuming $H$ is connected reductive, the Langlands dual group of $H$ is denoted by $H^\vee$, which is a connected reductive $\CC$-group on which $\Gal_F$ acting by pinned automorphisms. We shall identify $H^\vee$ with the group of its $\CC$-points. The set of equivalence classes of L-parameters is denoted by $\Phi(H)$; see \S\ref{sec:L-parameters} for details.

Irreducible representations of $\GL(n, F)$ are identified with the corresponding $n$-dimensional representation of $\mathcal{L}_F$ by LLC. If $\rho$ is a representation of $\GL(n, F)$ and $x \in \CC$, we put $\rho|\cdot|^x := \rho \otimes |\det|^x$.

A symplectic (resp.\ quadratic) $F$-vector space is defined to be a finite-dimensional $F$-vector space $W$ (resp.\ $V$) equipped with a non-degenerate alternating (resp.\ symmetric) bilinear form. Notations like $\Sp(W)$ (resp.\ $\SO(V) \subset \Or(V)$) indicate the corresponding symplectic (resp.\ orthogonal) groups, which will often be identified with the groups of $F$-points for notational convenience.

\section{Review of Hecke algebra correspondences}\label{sec:review}
\subsection{The metaplectic covering}\label{sec:Mp}
Let $F$ be a non-Archimedean local field of characteristic zero. Throughout this article, we fix an additive character $\bpsi$ of $F$ such that
\[ \bpsi|_{4\mathfrak{o}} = \mathbf{1}, \quad \bpsi|_{4\varpi^{-1}\mathfrak{o}} \not\equiv \mathbf{1}. \]

Consider a symplectic $F$-vector space $(W, \lrangle{\cdot|\cdot})$ with $\dim W = 2n > 0$, and choose a symplectic basis
\[ e_1, \ldots, e_n, f_n, \ldots, f_1, \quad \lrangle{e_i | f_j} = \delta_{i, j}. \]

Let $G := \Sp(W)$ be the associated symplectic group. The twofold metaplectic covering of $G(F)$ is the non-split central extension
\[ 1 \to \bmu_2 \to \tilde{G}^{(2)} \xrightarrow{\rev^{(2)}} G(F) \to 1 \]
of locally compact groups, which is unique up to unique isomorphism. We also denote $\tilde{G}^{(2)}$ as $\MMp(W)$ or $\MMp(2n)$.

In this article, $\tilde{G}$ is taken to be the push-out of $\tilde{G}^{(2)}$ along $\bmu_2 \hookrightarrow \bmu_8$, also denoted as $\Mp(W)$ or $\Mp(2n)$, with the convention that $\Mp(0) := \bmu_8$. It fits into a central extension of locally compact groups
\[ 1 \to \bmu_8 \to \tilde{G} \xrightarrow{\rev} G(F) \to 1. \]
We say that $\tilde{G}$ is a metaplectic group of rank $n$.

The group $G(F)$ acts on $\tilde{G}$ and $\tilde{G}^{(2)}$ by conjugation $\tilde{x} \mapsto g\tilde{x}g^{-1}$, by choosing any representative $\tilde{g}$ of $g \in G(F)$. An element $\tilde{\delta}$ of $\tilde{G}$ is said to be regular semisimple if $\delta := \rev(\tilde{\delta}) \in G(F)$ is.

Using $\lrangle{\cdot|\cdot}$ and $\bpsi$, the Weil representation $\omega_{\bpsi} = \omega_{\bpsi}^+ \oplus \omega_{\bpsi}^-$ (the even and odd pieces) can be defined on the level of $\tilde{G}$. Schrödinger models for $\omega_\bpsi$ furnish a $2$-cocycle describing $\tilde{G}$. Since Weil's index $\gamma_{\bpsi}(\cdot)$ for quadratic $F$-vector spaces is $\bmu_8$-valued, the cocycle for $\tilde{G}$ has a simpler form than that for $\tilde{G}^{(2)}$ due to Rao \cite{Rao93} or Lion--Perrin. For an overview, see \cite[\S 2]{Li11}.

Suppose that a Haar measure $\mu$ on $G(F)$ is chosen. We obtain a Haar measure $\tilde{\mu}$ on $\tilde{G}$ such that $\tilde{\mu}(\rev^{-1}(E)) = \mu(E)$ for all measurable subsets $E \subset G(F)$.

\begin{itemize}
	\item A representation $\pi$ of $\tilde{G}$ is said to be \emph{genuine} if $\pi(z) = z \cdot \identity$ for all $z \in \bmu_8$. For instance, $\omega_{\bpsi}^{\pm}$ are both genuine.
	\item A function $f: \tilde{G} \to \CC$ is said to be \emph{anti-genuine} if $f(z\tilde{x}) = z^{-1} f(\tilde{x})$ for all $z \in \bmu_8$.
	\item Define $C^\infty_{c, \asp}(\tilde{G})$ as the space of anti-genuine $C^\infty_c$-functions on $\tilde{G}$, and let $\mathcal{I}_{\asp}(\tilde{G})$ be its quotient by the subspace of $f \in C^\infty_{c, \asp}(\tilde{G})$ whose regular semisimple orbital integrals are all zero.
	\item Let $D_-(\tilde{G})$ be the linear dual of $\mathcal{I}_{\asp}(\tilde{G})$. When the Haar measure is chosen, $D_-(\tilde{G})$ is the space of genuine invariant distributions on $\tilde{G}$. 
\end{itemize}

The definitions above pertain to $\tilde{G}^{(2)}$ as well. As $\tilde{G}$ is a push-out of $\tilde{G}^{(2)}$, the genuine representations (resp.\ genuine invariant distributions) of $\tilde{G}$ and $\tilde{G}^{(2)}$ are identified.

With the chosen $\bpsi$ and $\lrangle{\cdot|\cdot}$, define the Langlands dual group of $\tilde{G}$ as
\[ \tilde{G}^\vee := \Sp(2n, \CC), \;\text{with trivial Galois action.} \]
It comes equipped with a standard pinning.

\subsection{Splitting over Levi subgroups}\label{sec:splitting-Levi}
Via the chosen symplectic basis for $(W, \lrangle{\cdot|\cdot})$, we obtain a $\Z$-basis $\epsilon_1, \ldots, \epsilon_n$ of $X^*(T)$; it also yields the natural Borel pair $(B^{\rightarrow}, T)$ for $G$, so that the $B^{\rightarrow}$-simple roots are
\[ \epsilon_i - \epsilon_{i+1} \quad (1 \leq i < n), \quad 2 \epsilon_n. \]
We also need the ``reversed'' Borel pair $(B^{\leftarrow}, T)$ with $B^{\leftarrow}$-simple roots
\[ \epsilon_{i+1} - \epsilon_i \quad (1 \leq i < n), \quad 2\epsilon_1 . \]

A Levi subgroup $M$ of $G$ is said to be semi-standard if $M \supset T$. A parabolic subgroup $P$ of $G$ is said to be semi-standard if it has a semi-standard Levi factor. Write $\tilde{M} = \rev^{-1}(M(F))$ and $\tilde{P} = \rev^{-1}(P(F))$. It is a general property of coverings that whenever $P \subset G$ is a parabolic subgroup and $U$ is its unipotent radical, $\rev: \tilde{G} \to G(F)$ splits canonically over $U(F)$. Therefore, any Levi decomposition $P = MU$ gives $\tilde{P} = \tilde{M} U(F)$ inside $\tilde{G}$. This justifies the consideration of (normalized) parabolic induction and Jacquet functors, as an adjunction pair
\[\begin{tikzcd}
	r_{\tilde{P}} = r^{\tilde{G}}_{\tilde{P}}: \tilde{G}\dcate{Mod} \arrow[shift left, r] & \tilde{M}\dcate{Mod}: i^{\tilde{G}}_{\tilde{P}} = i_{\tilde{P}} \arrow[shift left, l]
\end{tikzcd}\]
where $\tilde{M}\dcate{Mod}$ (resp.\ $\tilde{G}\dcate{Mod}$) stands for the category of genuine smooth representations of $\tilde{M}$ (resp.\ $\tilde{G}$).

Every semi-standard Levi $M$ can be decomposed into
\[ M = \Sp(W^\flat) \times \prod_{k=1}^r \GL(n_k) \]
where $W^\flat$ is a symplectic subspace of $W$, possibly zero. One of the advantages of using eightfold coverings is that the metaplectic coverings split over the $\GL$-factors of $M(F)$: using the chosen $\bpsi$ and $\lrangle{\cdot|\cdot}$, there is a canonical isomorphism of central extensions of $M(F)$ by $\bmu_8$:
\begin{equation*}
	\tilde{M} \simeq \Mp(W^\flat) \times \prod_{k=1}^r \GL(n_k, F).
\end{equation*}
The isomorphism is characterized in terms of $\omega_{\bpsi}$. In particular, $\tilde{T}$ splits into $T(F) \times \bmu_8$, which is not always possible within the twofold covering $\tilde{G}^{(2)}$. We refer to \cite[\S 2]{Li11} for a detailed discussion.

\begin{definition}\label{def:P1}
	Let $P^1$ be the parabolic subgroup of $G$ containing $B^{\leftarrow}$ with Levi decomposition $P^1 = M^1 U^1$, where
	\begin{align*}
		M^1 & = \Sp(W^1) \times \GL(1)^{n-1} \supset T, \\
		W^1 & := Fe_1 \oplus Ff_1 \subset W.
	\end{align*}
	Let $T^1$ be the $\GL(1)^{n-1}$ factor in $M^1$, so that $T = \GL(1) \times T^1$ and $\tilde{M}^1 \simeq \Mp(W^1) \times T^1(F)$. Also, denote by $\omega_{\bpsi}^1 = \omega_{\bpsi}^{1, +} \oplus \omega_{\bpsi}^{1, -}$ the Weil representation of $\Mp(W^1)$.
\end{definition}

\subsection{The number \texorpdfstring{$t(w)$}{t(w)}}\label{sec:tw}
The $B^{\leftarrow}$-simple roots in $X^*(T)$ are
\[ \beta_1 = 2\epsilon_1, \quad \beta_2 = \epsilon_2 - \epsilon_1, \quad \ldots, \quad \beta_n = \epsilon_n - \epsilon_{n-1}. \]
Denote by $t_1, \ldots, t_n$ the reflections relative to $\beta_1, \ldots, \beta_n$. They generate the Weyl group $\Omega^G_0$ of $G$ relative to $T$.

Identify $\Omega^G_0$ with the group $(\Z/2\Z)^n \rtimes \mathfrak{S}_n$ of ``signed permutations'' on $n$ letters. To each $w \in \Omega^G_0$, one attaches a number $t(w) \in \Z_{\geq 0}$ with the following interpretations.
\begin{itemize}
	\item First of all, $t(w)$ is defined in \cite[Definition 8.1.3]{CL23} as the number of $i$'s such that $w(2\epsilon_i)$ is a $B^{\leftarrow}$-negative root.
	\item Equivalently, $t(w)$ is the number of components $1 + 2\Z \in \Z/2\Z$ in $w$ as an element of $(\Z/2\Z)^n \rtimes \mathfrak{S}_n$.
	\item Take any reduced expression $w = t_{i_1} \cdots t_{i_\ell}$. By \cite[Lemma 8.1.4]{CL23}, $t(w)$ is the number of $k$'s such that $i_k = 1$.
\end{itemize}

The symplectic basis gives rise to a standard $F$-pinning of $G$. Given $\beta = \beta_i$, the homomorphism $x_\beta: \Ga \to G$ induced by the $F$-pinning lifts canonically to $\tilde{x}_\beta: F \to \tilde{G}$.

For each $1 \leq i \leq n$, one defines in \cite[\S 8.1]{CL23} a representative $\tilde{t}_i \in \tilde{G}$ as follows.
\begin{itemize}
	\item For $i > 1$, one repeats the Langlands--Shelstad construction to define
	\[ \tilde{t}_i := \tilde{x}_{\beta_i}(1) \tilde{x}_{-\beta_i}(-1) \tilde{x}_{\beta_i}(1) \in \tilde{G}^{(2)} \subset \tilde{G}. \]
	\item For $i = 1$, one multiplies $\tilde{x}_{\beta_1}(1) \tilde{x}_{-\beta_1}(-1) \tilde{x}_{\beta_1}(1)$ by a suitable element of $\bmu_8$ to obtain $\tilde{t}_1$, so that $\omega_{\bpsi}(\tilde{t}_1)$ is the unitary Fourier transform in the first coordinate in the Schrödinger model.
\end{itemize}

Set $\tilde{w} := \tilde{t}_{i_1} \cdots \tilde{t}_{i_\ell} \in \tilde{G}$. By \cite[Lemma 8.1.1]{CL23}, it is independent of the choice of reduced expression. By construction, $\tilde{w} \in \rev^{-1}(K_0)$ where $K_0 := G(\mathfrak{o})$.

Ishimoto defined another representative $\tilde{w}_{\mathrm{Ishi}} \in \tilde{G}^{(2)}$ of $w$ in \cite[p.1573]{Is20}. The only difference is that when $i=1$, he does not modify $\tilde{x}_{\beta_1}(1) \tilde{x}_{-\beta_1}(-1) \tilde{x}_{\beta_1}(1)$ by $\bmu_8$. The resulting $\tilde{w}_{\mathrm{Ishi}}$ is still independent of reduced expressions. The following comparison between $\tilde{w}$ and $\tilde{w}_{\mathrm{Ishi}}$ will be used in \S\ref{sec:application-LIR}. Define Weil's constant $\gamma_F(\bpsi) \in \bmu_8$ as in \cite[Appendix A]{Rao93}.

\begin{proposition}\label{prop:gamma-computation}
	For all $w \in \Omega^G_0$ we have
	\[ \tilde{w} = \gamma_F(\bpsi)^{-t(w)} \tilde{w}_{\mathrm{Ishi}}. \]
\end{proposition}
\begin{proof}
	Take a reduced expression
	\[ w = t_{i_1} \cdots t_{i_\ell}, \quad 1 \leq i_1, \ldots, i_\ell \leq n. \]
	Let us compare the representatives of $t_i$ in $\tilde{G}$ for each $i$. If $i > 1$, they are equal; if $i=1$ then they differ by some element of $\bmu_8$. Since $t(w)$ is the number $k$'s with $i_k = 1$, the assertion reduces easily to the case where $n=1$ and $w = t_1$.
	
	In the particular case above, put $\beta := \beta_1$. There exists $\xi \in \bmu_8$ such that
	\[ \tilde{t}_1 = \xi \tilde{x}_\beta(1) \tilde{x}_{-\beta}(-1) \tilde{x}_\beta(1). \]
	To determine $\xi$, we realize $\omega_{\bpsi}$ on the Schwartz space $\mathcal{S}(F)$, with the convention of \cite[\S 3.4]{CL23}. Denote by $\mathbf{1}_E \in \mathcal{S}(F)$ the indicator function of any subset $E \subset F$.
	
	Let us write $a \sim b$ to mean that two functions or numbers $a$ and $b$ agree up to $\R^{\times}_{> 0}$. Since $\omega_{\bpsi}(\tilde{t}_1)$ is the unitary Fourier transform, $\omega_{\bpsi}(\tilde{t}_1)(\mathbf{1}_{2\mathfrak{o}}) \sim \mathbf{1}_{\mathfrak{o}}$ by \cite[Lemma 8.2.3]{CL23} and vice versa. On the other hand, $\omega_{\bpsi}(\tilde{x}_\beta(1)) \mathbf{1}_{2\mathfrak{o}} = \mathbf{1}_{2\mathfrak{o}}$ by \cite[(3.4)]{CL23}. Hence
	\begin{equation}\label{eqn:gamma-computation}
		\mathbf{1}_{\mathfrak{o}} \sim \xi \cdot \omega_{\bpsi}(\tilde{x}_\beta(1)) \omega_{\bpsi}(\tilde{x}_{-\beta}(-1)) \mathbf{1}_{2\mathfrak{o}}.
	\end{equation}

	To study the right hand side of \eqref{eqn:gamma-computation}, observe that $\tilde{x}_{-\beta}(-1)$ as $\tilde{t}_1^{-1} \tilde{x}_\beta(1) \tilde{t}_1$. Then
	\begin{align*}
		\omega_{\bpsi}(\tilde{x}_{-\beta}(-1)) \mathbf{1}_{2\mathfrak{o}} & \sim \omega_{\bpsi}(\tilde{t}_1)^{-1} \omega_{\bpsi}(\tilde{x}_\beta(1))(\mathbf{1}_{\mathfrak{o}}) \\
		& \sim \omega_{\bpsi}(\tilde{t}_1)^{-1} \left[ x \mapsto \bpsi(x^2) \mathbf{1}_{\mathfrak{o}}(x) \right]
	\end{align*}
	by \cite[(3.4)]{CL23}, whose value at $0$ is $\sim \int_{\mathfrak{o}} \bpsi(y^2) \dd y$ since $\omega_{\bpsi}(\tilde{t}_1)$ is the unitary Fourier transform.
	
	By \cite[Theorem A.3]{Rao93}, $\int_{\mathfrak{o}} \bpsi(y^2) \dd y$ reduces to a Gauss sum over $\mathfrak{o}/2\mathfrak{o}$ modulo $\sim$, since
	\[ \frac{\bpsi((y+z)^2)}{\bpsi(y^2) \bpsi(z^2)} = \bpsi(2yz), \quad \forall y \in 2\mathfrak{o}, \; \bpsi(y^2) = 1, \]
	and the Gauss sum yields $\gamma_F(\bpsi)$.
	
	By \cite[(3.4)]{CL23}, the value at $0$ of functions in $\mathcal{S}(F)$ is not affected by $\omega_{\bpsi}(\tilde{x}_\beta(1))$. Comparing the values at $0$ of both sides of \eqref{eqn:gamma-computation}, we see $\xi \gamma_F(\bpsi) \sim 1$, or equivalently $\xi = \gamma_F(\bpsi)^{-1}$.
\end{proof}

\subsection{Bernstein blocks and the Takeda--Wood isomorphism}\label{sec:blocks}
Inside $\tilde{G}\dcate{Mod}$, denote by $\mathcal{G}_{\bpsi}^\pm$ the Bernstein blocks containing the irreducible representations $\omega_{\bpsi}^{\pm}$. They are abelian subcategories of $\tilde{G}\dcate{Mod}$.

On the other hand, take $V^{\pm}$ to be the quadratic $F$-vector space of dimension $2n+1$, discriminant $1$ and Hasse invariant $\pm 1$. Set
\[ G^{\pm} := \SO(V^{\pm}) \]
and let $\mathcal{G}^\pm$ be the Bernstein block containing the trivial representation $\mathbf{1}_{G^{\pm}(F)}$ of $G^{\pm}(F)$, i.e.\ the Iwahori-spherical block in $G^{\pm}\dcate{Mod}$. We refer to \cite{GS2, TW18} for further discussions.

Fix an Iwahori subgroup of $G^{\pm}(F)$, and let $H^{\pm}$ be the Iwahori--Hecke algebra of $G^{\pm}(F)$. There is an equivalence of categories
\[ \mathcal{G}^{\pm} \simeq H^{\pm}\dcate{Mod}. \]

By the results of Gan--Savin \cite{GS2} and Takeda--Wood \cite{TW18}, there is a similar picture for the blocks $\mathcal{G}^{\pm}_{\bpsi}$ for $\tilde{G}$. Specifically, using the Iwahori subgroup of $G(F)$ attached to the Borel pair $(B^{\rightarrow}, T)$, in both the $\pm$ cases one has
\begin{itemize}
	\item an explicitly presented $\CC$-algebra $H_{\bpsi}^{\pm}$,
	\item an equivalence $\mathcal{G}_{\bpsi}^{\pm} \simeq H_{\bpsi}^{\pm}\dcate{Mod}$ of categories, namely by constructing types,
	\item an explicit isomorphism of $\CC$-algebras $\mathrm{TW}: H^{\pm} \rightiso H^{\pm}_{\bpsi}$.
\end{itemize}
Both $H^{\pm}$ and $H_{\bpsi}^{\pm}$ have natural Hilbert algebra structures, which are preserved by the Takeda--Wood isomorphism $\mathrm{TW}$. For a summary, we refer to \cite[\S 1]{CL23}.

In the references cited above, it is also shown that:
\begin{itemize}
	\item the irreducibles in $\mathcal{G}^+_{\bpsi}$ are constituents of $i_{\tilde{B}^{\leftarrow}}(\alpha)$, where $\alpha$ is any unramified character of $T(F)$, also viewed a genuine character of $\tilde{T} \simeq T(F) \times \bmu_8$;
	\item the irreducibles in $\mathcal{G}^-_{\bpsi}$ are constituents of $i^{\tilde{G}}_{\tilde{P}^1}(\omega_{\bpsi}^{1, -} \boxtimes \beta)$, where $\beta$ is any unramified character of $T^1(F)$ (see Definition \ref{def:P1}).
\end{itemize}

\begin{definition}\label{def:TW-star}
	Denote by $\mathrm{TW}^*: \mathcal{G}_{\bpsi}^{\pm} \to \mathcal{G}^{\pm}$ the equivalence of categories induced by $\mathrm{TW}: H^{\pm} \rightiso H^{\pm}_{\bpsi}$.
\end{definition}

Since $\mathrm{TW}$ reserves Hilbert structures, $\mathrm{TW}^*$ preserves temperedness and square-integrability of irreducible representations.

\subsection{Compatibilities}\label{sec:compatibilities}
There are standard choices of minimal parabolic and Levi subgroups of $G^{\pm}$. However, it is convenient to use the ``reversed'' minimal parabolic $P_{\min}^{\pm} \subset G^{\pm}$, as in the setting of $G$ in \S\ref{sec:splitting-Levi}.

The Levi factors $M_{\min}^{\pm}$ of $P_{\min}^{\pm}$ are taken to be
\begin{align*}
	M_{\min}^+ & \simeq \GL(1)^n, \\
	M_{\min}^- & \simeq \SO(V_1^-) \times \GL(1)^{n-1},
\end{align*}
where $V_1^-$ is a $3$-dimensional quadratic $F$-vector space with discriminant $1$ and Hasse invariant $-1$. The groups of unramified characters of $M_{\min}^+(F)$ and $T(F)$ (resp.\ $M_{\min}^-(F)$ and $M^1(F)$) are canonically matched.

There are natural bijections
\begin{equation}\label{eqn:P-corr}
	\begin{tikzcd}[row sep=tiny]
		\{ \text{parabolic subgroups}\; B^{\leftarrow} \subset P \subset G \} \arrow[leftrightarrow, r, "1:1"] & \{ \text{parabolic subgroups}\; P_{\min}^+ \subset P^+ \subset G^+ \}, \\
		\{ \text{parabolic subgroups}\; P^1 \subset P \subset G \} \arrow[leftrightarrow, r, "1:1"] & \{ \text{parabolic subgroups}\; P_{\min}^- \subset P^- \subset G^- \},
	\end{tikzcd}
\end{equation}
written as $P \leftrightarrow P^{\pm}$. Consequently, there are natural bijections between certain Levi subgroups $M \leftrightarrow M^{\pm}$, such that if
\[ M = \Sp(W^\flat) \times \prod_{k=1}^r \GL(n_k), \]
then
\[ M^{\pm} = \SO(V^{\pm, \flat}) \times \prod_{k=1}^r \GL(n_k) \]
where $V^{\pm, \flat}$ is a quadratic subspace of $V^{\pm}$ with dimension equal to $\dim W^{\flat} + 1$. Note that $M$ corresponds to some $M^-$ if and only if $\dim W^{\flat} > 0$.
\begin{itemize}
	\item In particular, the groups of unramified characters of $M(F)$ and $M^\pm(F)$ are identified: both reduce to the group of unramified characters of $\prod_{k=1}^r \GL(n_k, F)$. These unramified characters will be used to twist genuine representations of $\tilde{M}$ (resp.\ representations of $M^{\pm}(F)$).
	\item Denote by $A_M$ (resp.\ $A_{M^{\pm}}$) the center of $M$ (resp.\ $M^\pm$). Then $A_M$ and $A_{M^\pm}$ are also identified: both equal $\Gm^r$. The ``roots'' of $A_M$ acting on the unipotent radical of the parabolic match those of $A_{M^\pm}$, up to some factor $2$.
\end{itemize}

For $M \leftrightarrow M^{\pm}$ as above, define the algebras
\begin{equation}\label{eqn:H-decomp}
	\begin{aligned}
		H_{\bpsi}^{\tilde{M}, \pm} & := H_{\bpsi}^{\pm, \flat} \otimes \bigotimes_{k=1}^r H_k, \\
		H^{M^{\pm}} & := H^{\pm, \flat} \otimes \bigotimes_{k=1}^r H_k,
	\end{aligned}
\end{equation}
where
\begin{itemize}
	\item $H_{\bpsi}^{\pm, \flat}$ is the avatar of $H_{\bpsi}^{\pm}$ for $\Mp(W^\flat)$;
	\item $H^{\pm, \flat}$ is the avatar of $H^{\pm}$ for $\SO(V^{\pm, \flat})$;
	\item $H_k$ is the Iwahori--Hecke algebra of $\GL(n_k)$ for all $1 \leq k \leq r$, defined relative to the standard Borel pair.
\end{itemize}

Define the avatar $\mathcal{G}_{\bpsi}^{\tilde{M}, \pm}$ (resp.\ $\mathcal{G}^{M^{\pm}}$) of $\mathcal{G}_{\bpsi}{^\pm}$ (resp.\ $\mathcal{G}^{\pm}$) for $\tilde{M}$ (resp.\ $M^\pm$) in the similar way, so that we have equivalences $\mathcal{G}_{\bpsi}^{\tilde{M}, \pm} \simeq H_{\bpsi}^{\tilde{M}, \pm}\dcate{Mod}$ and $\mathcal{G}^{M^{\pm}} \simeq H^{M^{\pm}}\dcate{Mod}$.

Let $\mathrm{TW}^\flat: H^{\pm, \flat} \rightiso H_{\bpsi}^{\pm, \flat}$ be the isomorphism defined relative to $W^\flat$ and $V^{\pm, \flat}$. From this we obtain
\begin{equation}\label{eqn:TW-M}
	\mathrm{TW}^{\tilde{M}} := \mathrm{TW}^{\flat} \otimes \bigotimes_{k=1}^r \identity_{H_k}: H^{M^{\pm}} \rightiso H^{\tilde{M}, \pm}_{\bpsi}.
\end{equation}

Suppose $P \leftrightarrow P^{\pm}$ under \eqref{eqn:P-corr}, thus $M \leftrightarrow M^{\pm}$. By \cite[\S 7]{BK98}, there is a canonical embedding $\mathrm{t}^{\pm}_{\mathrm{nor}}: H^{M^{\pm}} \to H^{\pm}$ of Hilbert algebras, through which the functor $r_{P^{\pm}}: \mathcal{G}^{\pm} \to \mathcal{G}^{M^{\pm}}$ matches $(\mathrm{t}^{\pm}_{\mathrm{nor}})^*: H^{\pm}\dcate{Mod} \to H^{M^{\pm}}\dcate{Mod}$. By adjunction, $i_{P^{\pm}}: \mathcal{G}^{M^{\pm}} \to \mathcal{G}^{\pm}$ matches $\Hom_{H^{M^{\pm}}}(H^{\pm}, \cdot)$.

On the metaplectic side, define the embedding
\[ \mathrm{t}_{\mathrm{nor}} := \text{the composite of}\; H_{\bpsi}^{\tilde{M}, \pm} \xrightarrow[\sim]{(\mathrm{TW}^{\tilde{M}})^{-1}} H^{M^{\pm}} \xrightarrow{\mathrm{t}^{\pm}_{\mathrm{nor}}} H^{\pm} \xrightarrow[\sim]{\mathrm{TW}} H_{\bpsi}^{\pm}. \]

We are ready to summarize the main results of \cite{CL23}: compatibility between the Takeda--Wood isomorphism $\mathrm{TW}$ and Jacquet module, parabolic induction, and standard (i.e.\ unnormalized) intertwining operators.

\begin{theorem}[{\cite[Theorem 1.1]{CL23}}]
	\label{prop:compatibility-1}
	Let $P \supset B^{\leftarrow}$ (the $+$ case) or $P \supset P^1$ (the $-$ case) be a parabolic subgroup of $G$, with Levi decomposition $MU$ with $M \supset T$. Under $\mathrm{TW}^*$ and $\mathrm{TW}^{\tilde{M}, *}$, the adjunction pair
	\[\begin{tikzcd}
		r_{\tilde{P}}: \mathcal{G}_{\bpsi}^{\pm} \arrow[shift left, r] & \mathcal{G}_{\bpsi}^{\tilde{M}, \pm} : i_{\tilde{P}} \arrow[shift left, l] 
	\end{tikzcd}\]
	corresponds to
	\[\begin{tikzcd}
		\mathrm{t}_{\mathrm{nor}}^*: H_{\bpsi}^{\pm}\dcate{Mod} \arrow[shift left, r] & H_{\bpsi}^{\tilde{M}, \pm}\dcate{Mod} : \Hom_{H_{\bpsi}^{\tilde{M}, \pm}}(H_{\bpsi}^{\pm}, \cdot). \arrow[shift left, l]
	\end{tikzcd}\]
	Simply put, the equivalences $\mathrm{TW}^*$ and $\mathrm{TW}^{\tilde{M}, *}$ respect Jacquet modules and parabolic inductions.
\end{theorem}

For a Levi subgroup $M$ of $G$ and $\tilde{w} \in \tilde{G}$ that normalizes $\tilde{M}$, define the twist ${}^{\tilde{w}} \pi$ for any genuine representation $\pi$ of $\tilde{M}$ by
\[ \left( {}^{\tilde{w}} \pi \right)(\tilde{g}) = \pi(\tilde{w}^{-1} \tilde{g} \tilde{w}), \quad \tilde{g} \in \tilde{G}, \]
with the same underlying vector space. Ditto for $M^{\pm}(F) \subset G^{\pm}(F)$.

\begin{theorem}[{\cite[Theorem 1.3 + Remark 10.2.7]{CL23}}]
	\label{prop:compatibility-2}
	Let $P = MU$ be as above. In the $+$ case, suppose that $P \leftrightarrow P^+$ via \eqref{eqn:P-corr}, thus $M \leftrightarrow M^+$. Let $w \in \Omega^G_0$ be such that $wMw^{-1} = M$ and $w$ has minimal length in the $\Omega^M_0$-coset containing $w$, where the length is computed relative to the $B^{\leftarrow}$-simple roots. Take the representative $\tilde{w}$ of $w$ in $\tilde{G}$ defined in \S\ref{sec:tw}; take the Langlands--Shelstad representative $\dot{w}$ in $G^{\pm}(F)$ (see Definition \ref{def:Rpm-operator} and the references therein).
	
	Let $\pi$ (resp.\ $\sigma$) be of finite length in $\mathcal{G}_{\bpsi}^{\tilde{M}, +}$ (resp.\ $\mathcal{G}^{M^+}$) and fix a $\mathrm{TW}^{\tilde{M}}$-equivariant isomorphism between their Hecke modules. Consider the standard intertwining operators
	\begin{align*}
		J_w(\pi \otimes \alpha): i_{\tilde{P}}(\pi \otimes \alpha) & \to i_{\tilde{P}}({}^{\tilde{w}}(\pi \otimes \alpha)), \\
		J^+_w(\sigma \otimes \alpha): i_{P^+}(\sigma \otimes \alpha) & \to i_{P^+}({}^{\dot{w}}(\sigma \otimes \alpha))
	\end{align*}
	viewed as rational families in unramified characters $\alpha$ of $\prod_{k=1}^r \GL(n_k, F)$; they satisfy
	\[ J_w(\pi \otimes \alpha) \xleftrightarrow{\text{matches}} |2|_F^{t(w)/2} J^+_w(\sigma \otimes \alpha) \]
	through $\mathrm{TW}^{\tilde{M}}$ and $\mathrm{TW}$ on the level of Hecke modules.
	
	In the $-$ case, when $P \leftrightarrow P^-$ via \eqref{eqn:P-corr} and $M \leftrightarrow M^-$, define $J_w^-(\sigma \otimes \alpha)$ analogously; we have
	\[ J_w(\pi \otimes \alpha) \xleftrightarrow{\text{matches}} (-q^{-1})^{t(w)} |2|_F^{t(w)/2} J^-_w(\sigma \otimes \alpha). \]
	Caveats:
	\begin{itemize}
		\item the implicated Haar measures in $J_w$, $J_w^{\pm}$ are normalized using suitable special maximal compact subgroups $K_0$, $K^\pm$ respectively;
		\item $J_w$ (resp.\ $J_w^\pm$) is defined using the representative $\tilde{w}$ (resp.\ $\dot{w}$);
		\item the isomorphism between the Hecke modules of $(\pi, \sigma)$ and that between their Weyl-twists must be chosen compatibly.  
	\end{itemize}
\end{theorem}

The precise recipe will be recalled whenever needed.

\section{Statement of the main result}\label{sec:statement}
\subsection{Representations and L-parameters}\label{sec:L-parameters}
Let $\Pi_-(\tilde{G})$ be the set of isomorphism classes of irreducible genuine smooth representations of $\tilde{G}$. Let $\Pi_{\mathrm{temp}, -}(\tilde{G})$ (resp.\ $\Pi_{2, -}(\tilde{G})$) be its subset consisting of tempered (resp.\ square-integrable) irreducibles.

With the definition that $\tilde{G}^\vee = \Sp(2n, \CC)$, denote by $\Phi(\tilde{G})$ the set of equivalence classes of L-parameters for $\tilde{G}$, and let $\Phi_{\mathrm{bdd}}(\tilde{G})$ (resp.\ $\Phi_{2, \mathrm{bdd}}(\tilde{G})$) be its subset consisting of bounded (resp.\ discrete series) L-parameters.

For every connected reductive $F$-group $H$, define the sets $\Pi_2(H) \subset \Pi_{\mathrm{temp}}(H) \subset \Pi(H)$ of irreducible smooth representations of $H(F)$, as well as the sets $\Phi_{2, \mathrm{bdd}}(H) \subset \Pi_{\mathrm{bdd}}(H) \subset \Phi(H)$ of L-parameters in the same manner.

For each $\phi \in \Phi(\tilde{G})$, define its centralizer and component groups as
\[ S_\phi := Z_{\tilde{G}^\vee}(\Image(\phi)), \quad \EuScript{S}_\phi := \pi_0(S_\phi). \]

Note that $\Phi(\tilde{G}) = \Phi(\SO(2n+1))$. View $\phi \in \Phi(\tilde{G})$ as a $2n$-dimensional representation of $\mathcal{L}_F$; here the representations are assumed to be continuous, Frobenius-semisimple on the Weil group of $F$, and algebraic on the $\SL(2, \CC)$ factor of $\mathcal{L}_F$. Then there is a unique decomposition into simple summands
\[ \phi = \bigoplus_{i \in I} m_i \phi_i, \quad m_i \in \Z_{\geq 1}, \quad i \neq i' \implies \phi_i \not\simeq \phi_{i'}, \]
where the indexing set $I$ is finite.

Furthermore, we may decompose $I = I^+ \sqcup I^- \sqcup J \sqcup J'$, such that $J$ and $J'$ are related by a bijection $j \leftrightarrow j'$, and
\begin{itemize}
	\item $\phi_i$ is self-dual of symplectic (resp.\ orthogonal) type when $i \in I^+$ (resp.\ $i \in I^-$);
	\item $\phi_j$ is not self-dual when $j \in J$;
	\item $m_i$ is even for all $i \in I^-$;
	\item $\phi_{j'} \simeq \check{\phi}_j$ and $m_j = m_{j'}$ when $j \leftrightarrow j'$, where $\check{\phi}_j$ denotes the contragredient of $\phi_j$.
\end{itemize}

The subsets $I^{\pm}$ are uniquely determined, but $J$ is not. We have
\begin{itemize}
	\item $\phi \in \Phi_{\mathrm{bdd}}(\tilde{G})$ if and only if $\phi_j$ is bounded for each $j \in J$;
	\item $\phi \in \Phi_{2, \mathrm{bdd}}(\tilde{G})$ if and only if $I = I^+$ and $m_i = 1$ for all $i$;
	\item we have
	\begin{equation}\label{eqn:S}\begin{aligned}
		S_\phi & \simeq \prod_{i \in I^+} \Or(m_i, \CC) \times \prod_{i \in I^-} \Sp(m_i, \CC) \times \prod_{j \in J} \GL(m_j, \CC), \\
		\EuScript{S}_\phi & \simeq \prod_{i \in I^+} \bmu_2,
	\end{aligned}\end{equation}
	under which the map $S_\phi \to \EuScript{S}_\phi$ is given by determinants on the orthogonal factors.
\end{itemize}
In particular, $\EuScript{S}_\phi$ is finite abelian and we denote its Pontryagin dual as $\EuScript{S}_\phi^\vee$.

We will also need the notion of Jordan blocks in $\phi$.

\begin{definition}\label{def:Jord}
	For every $\phi \in \Phi(\tilde{G})$, let $\mathrm{Jord}(\phi)$ be the multi-set of pairs $(\rho, a)$ in the decomposition of $\phi$ into a sum of representations $\rho \boxtimes r(a)$ of $\mathcal{L}_F$, where $\rho$ is a supercuspidal representation of some $\GL(d_\rho, F)$, corresponding to a simple representation of $\Weil{F}$ of dimension $d_\rho \geq 1$, and $r(a)$ is the $a$-dimensional irreducible representation of $\SL(2, \CC)$.
\end{definition}

If $\phi \in \Phi_{2, \mathrm{bdd}}(\tilde{G})$, then $\mathrm{Jord}(\phi)$ is multiplicity-free. Given $(\rho, a) \in \mathrm{Jord}(\phi)$, we shall write $\chi(\rho, a)$ for the value of $\chi \in \EuScript{S}_\phi^\vee$ on $-1 \in \bmu_2$ in the direct factor of $\EuScript{S}_\phi$ attached to $\rho \boxtimes r(a)$.

\subsection{Local Langlands correspondence}\label{sec:LLC}
With our choice of $\bpsi$ and $(W, \lrangle{\cdot|\cdot})$, the local Langlands correspondence for $\tilde{G}$ can be phrased as a canonical decomposition
\begin{equation}\label{eqn:LLC-Mp}
	\begin{aligned}
		\Pi_-(\tilde{G}) & = \bigsqcup_{\phi \in \Phi(\tilde{G})} \Pi_\phi, \\
		\EuScript{S}_\phi^\vee & \xleftrightarrow{1:1} \Pi_\phi \\
		\chi & \longmapsto \pi_{\phi, \chi}.
	\end{aligned}
\end{equation}

Moreover, the decomposition in \eqref{eqn:LLC-Mp} restricts to
\begin{align*}
	\Pi_{\mathrm{temp}, -}(\tilde{G}) & = \bigsqcup_{\phi \in \Phi_{\mathrm{bdd}}(\tilde{G})} \Pi_\phi , \\
	\Pi_{2, -}(\tilde{G}) & = \bigsqcup_{\phi \in \Phi_{2, \mathrm{bdd}}(\tilde{G})} \Pi_\phi.
\end{align*}

In \cite{GS1}, the correspondence \eqref{eqn:LLC-Mp} is established via $\Theta$-lifting and the LLC for special odd orthogonal groups, which is to be reviewed in \S\ref{sec:main}. Furthermore, \eqref{eqn:LLC-Mp} is shown to be compatible with Langlands quotients in \textit{loc.\ cit.} The precise meaning of compatibility is reviewed below.

Take any Borel subgroup $B$ containing $T$. For each parabolic subgroup $P \supset B$ of $G$ with Levi factor $M = G^\flat \times \prod_{k=1}^r \GL(n_k) \supset T$, where $G^\flat = \Sp(W^\flat)$ with $W^\flat$ being a symplectic subspace of $W$, the dual group of the covering
\begin{equation}\label{eqn:metaplectic-type}
	\tilde{M} = \tilde{G}^\flat \times \prod_{k=1}^r \GL(n_k, F) \xrightarrow{(\rev, \identity)} G^\flat(F) \times \prod_{k=1}^r \GL(n_k, F)
\end{equation}
is defined to be
\[ \tilde{M}^\vee := (\tilde{G}^\flat)^\vee \times \prod_{k=1}^r \GL(n_k, \CC) \]
with trivial Galois action, so that $\tilde{M}^\vee \hookrightarrow \tilde{G}^\vee$ canonically up to Weyl group actions. Define $\Phi(\tilde{M})$, etc.\ accordingly, then there is a natural map
\begin{equation}\label{eqn:Phi-M-G}
	\Phi(\tilde{M}) \to \Phi(\tilde{G}).
\end{equation}
Combining \eqref{eqn:LLC-Mp} for $\tilde{G}^\flat$ and the LLC for general linear groups, we obtain \eqref{eqn:LLC-Mp} for $\tilde{M}$.

Given $\phi \in \Phi(\tilde{G})$, we may find a unique $P \supset B$ above, together with a unique decomposition
\[ \phi = \phi_0 \oplus \left( \phi_1 \oplus \check{\phi}_1 \right), \]
such that
\begin{itemize} 
	\item $\phi_0 \in \Phi_{\mathrm{bdd}}(\tilde{G}^\flat)$;
	\item $\phi_1$ is an L-parameter for $\prod_{k=1}^r \GL(n_k)$, which is a twist of a bounded one by an $\R_{> 0}^{\times}$-valued unramified character that is $P$-positive (in the standard sense).
\end{itemize}
Thus $\phi$ is the image of $\phi_{\tilde{M}} = (\phi_0, \phi_1) \in \Phi(\tilde{M})$. Define centralizer and component groups for $\phi_{\tilde{M}}$ in the evident way. The summands in $\phi_1 \oplus \check{\phi}_1$ do not affect the indexing set $I^+$, hence
\[ \EuScript{S}_{\phi_0} = \EuScript{S}_{\phi_{\tilde{M}}} \rightiso \EuScript{S}_\phi. \]

The tempered LLC for $\tilde{G}^\flat$ (resp.\ the LLC for $\prod_{k=1}^r \GL(n_k)$) furnishes the tempered L-packet $\Pi_{\phi_0}$ (resp.\ the representation $\pi_{\GL}$). Set $\Pi_{\phi_{\tilde{M}}} := \{\pi_0 \boxtimes \pi_{\GL} : \pi_0 \in \Pi_{\phi_0} \}$. Since $\pi_{\GL}$ is a $P$-positive twist of a tempered representation, we can define the L-packet
\begin{equation*}\begin{aligned}
	\Pi_\phi & := \left\{ J_{\tilde{P}}(\pi_0 \boxtimes \pi_{\GL}) : \pi_0 \in \Pi_{\phi_0} \right\} \\
	& = \left\{ J_{\tilde{P}}(\pi_{\tilde{M}}) : \pi_{\tilde{M}} \in \Pi_{\phi_{\tilde{M}}} \right\},
\end{aligned}\end{equation*}
where $J_{\tilde{P}}(\cdots)$ denotes the Langlands quotient of $i_{\tilde{P}}(\cdots)$. Since $\Pi_{\phi_0} \xleftrightarrow{1:1} \EuScript{S}_{\phi_0}^\vee$ by \eqref{eqn:LLC-Mp} for $\tilde{G}^\flat$, this yields the desired bijection
\begin{align*}
	\EuScript{S}_\phi^\vee & \xrightarrow{1:1} \Pi_\phi \\
	\chi & \mapsto J_{\tilde{P}}(\pi_{\phi_0, \chi} \boxtimes \pi_{\GL}).
\end{align*}

The procedure above is an instance of the general paradigm in \cite[\S 7.2]{SZ18}.

\subsection{Endoscopic character relations}\label{sec:ECR}
We refer to \cite[\S 3]{Li19} or \cite[\S 2]{Li24b} for a recapitulation of endoscopy for $\tilde{G}$. The following glossary serves to fix notations.

\begin{description}
	\item[Elliptic endoscopic data] These are conjugacy classes of $s \in \tilde{G}^\vee$ satisfying $s^2 = 1$, which are in bijection with pairs $(n', n'') \in \Z_{\geq 0}^2$ such that $n' + n'' = n$, corresponding to the eigenvalues $+1$ and $-1$ of $s$ respectively.
	
	For an elliptic endoscopic datum $\mathbf{G}^!$, the underlying endoscopic group is
	\[ G^! = \SO(2n'+1) \times \SO(2n''+1), \]
	and $(G^!)^\vee$ is identified with $Z_{\tilde{G}^\vee}(s) \subset \tilde{G}^\vee$. In particular, we get a map
	\begin{equation}\label{eqn:Phi-Endo}
		\Phi(G^!) \to \Phi(\tilde{G}).
	\end{equation}
	
	\item[Transfer of orbital integrals] There is a notion of transfer of orbital integrals from $\tilde{G}$ to $G^!$, denoted as $\Trans_{\mathbf{G}^!, \tilde{G}}$.
	
	\item[Transfer of distributions] The dual of $\Trans_{\mathbf{G}^!, \tilde{G}}$ is a linear map
	\[ \trans_{\mathbf{G}^!, \tilde{G}}: SD(G^!) \otimes \mes(G^!)^\vee \to D_-(\tilde{G}) \otimes \mes(G)^\vee \]
	where $SD_-(G^!)$ is the linear dual of
	\[ S\orbI(G^!) := \dfrac{C^\infty_c(G^!(F))}{\left\{ f^! : \forall \; \text{strongly regular semisimple orbital integral is zero} \right\}}; \]
	here a regular semisimple element of $G^!$ is said to be strongly regular if its centralizer is connected.
\end{description}
	
All the above depend on $\bpsi$ and the symplectic form $\lrangle{\cdot|\cdot}$ on $W$, which have been chosen.

\begin{definition}
	Denote by $\Endo_{\elli}(\tilde{G})$ the set of elliptic endoscopic data of $\tilde{G}$.
\end{definition}
	
Let $\phi \in \Phi_{\mathrm{bdd}}(\tilde{G})$ and suppose $s \in S_\phi$ satisfies $s^2 = 1$. Note that $s$ acts on the $2n$-dimensional representation $\mathrm{std} \circ \phi$ of $\mathcal{L}_F$, and denote the $(-1)$-eigenspace as $\phi^{s = -1}$. Set
\begin{equation}\label{eqn:epsilon-factor}
	\epsilon(\phi^{s = -1}) := \epsilon\left(\frac{1}{2}, \phi^{s = -1}, \bpsi \right);
\end{equation}
this is $\bmu_2$-valued since $\phi^{s = -1}$ is self-dual of symplectic type. We refer to \cite[\S 5]{GGP1} for the definition and basic properties for $\epsilon$-factors.

On the other hand, define
\begin{equation}\label{eqn:T}
	T_{\phi, s} := \epsilon(\phi^{s = -1}) \trans_{\mathbf{G}^!, \tilde{G}} \left( S\Theta^{G^!}_{\phi^!} \right),
\end{equation}
an element of $D_-(\tilde{G}) \otimes \mes(G)^\vee$, where
\begin{itemize}
	\item $\mathbf{G}^! \in \Endo_{\elli}(\tilde{G})$ corresponds to the conjugacy class of $s$;
	\item $\phi$ factors through $(G^!)^\vee$ accordingly, and gives rise to $\phi^! \in \Phi_{\mathrm{bdd}}(G^!)$ such that $\phi^! \mapsto \phi$ via \eqref{eqn:Phi-Endo};
	\item $S\Theta^{G^!}_{\phi^!} \in SD(G^!) \otimes \mes(G^!)^\vee$ is attached to $\phi^!$ by Arthur's theory \cite{Ar13}.
\end{itemize}

The ECR for $\tilde{G}$ is stated as follows.
	
\begin{theorem}[C.\ Luo \cite{Luo20}]\label{prop:Luo-endo}
	Let $\phi \in \Phi_{\mathrm{bdd}}(\tilde{G})$ and $s \in S_\phi$ with $s^2 = 1$. The $T_{\phi, s}$ in \eqref{eqn:T} depends only on the image $x \in \EuScript{S}_\phi$ of $s$. Every $x \in \EuScript{S}_\phi$ arises from some $s$ in this way, and $\pi_{\phi, \chi} \in \Pi_\phi$ are characterized by the following identities:
	\[ T_{\phi, x} = \sum_{\chi \in \EuScript{S}_\phi^\vee} \chi(x) \Tr\left( \pi_{\phi, \chi} \right), \quad x \in \EuScript{S}_\phi. \]
\end{theorem}

\begin{remark}\label{rem:Endo-M}
	Theorem \ref{prop:Luo-endo} extends to coverings of the form \eqref{eqn:metaplectic-type}, i.e.\ to Levi subgroups of metaplectic coverings. More precisely, recall that the LLC \eqref{eqn:LLC-Mp} has been extended to $\tilde{M}$; we define $\Endo_{\elli}(\tilde{M}) := \Endo_{\elli}(\tilde{G}^\flat)$, and the corresponding endoscopic group to be that of $\tilde{G}^\flat$ times $\prod_{k=1}^r \GL(n_k)$. The transfer map decomposes accordingly, being identity on the $\GL$ factors. The $\epsilon$-factor in \eqref{eqn:T} in this setting involves only the $\tilde{G}^\flat$-component of $\phi$.
\end{remark}

\subsection{Main theorem}\label{sec:main}
Decompose $\phi \in \Phi(\tilde{G})$ into $\displaystyle\bigoplus_{i \in I^+ \sqcup I^-} m_i \phi_i \oplus \displaystyle\bigoplus_{j \in J} m_j \phi_j$ as in \S\ref{sec:L-parameters}. Identify $\EuScript{S}_\phi$ with $\bmu_2^{I^+}$; identify $\EuScript{S}_\phi^\vee$ with $\bmu_2^{I^+}$ accordingly.

\begin{definition}\label{def:nu-phi}
	Let $\nu_\phi \in \EuScript{S}_\phi^\vee$ be the element whose component at each $i \in I^+$ is given by
	\[ \nu_{\phi, i} = \epsilon\left(\frac{1}{2}, \phi_i, \bpsi \right) \]
	as in \eqref{eqn:epsilon-factor}; we have $\nu_{\phi, i} \in \bmu_2$ since $\phi_i$ is self-dual of symplectic type.
\end{definition}

To state the basic result of this article, retain the definitions in \S\S\ref{sec:blocks}--\ref{sec:compatibilities}, and observe that $(G^{\pm})^\vee = \Sp(2n, \CC) = \tilde{G}^\vee$ with trivial Galois actions, so it makes sense to compare the L-parameters and centralizers. The LLC for $G^{\pm}$ from \cite{Ar13, Is24} is independent of the choice of Whittaker data since $G^{\pm}$ are adjoint groups. Following Vogan, we formulate the LLC as
\begin{equation}\label{eqn:LLC-SO}
	\Pi(G^+) \sqcup \Pi(G^-) = \bigsqcup_\phi \Pi^{G^{\pm}}_\phi
\end{equation}
where $\phi$ ranges over $\Phi(G^+) = \Phi(G^-) = \Phi(\tilde{G})$. For each $\phi$, defining $\EuScript{S}_\phi$ as in the case of $\tilde{G}$, we have
\begin{equation}\label{eqn:LLC-SO-chi}\begin{aligned}
	\EuScript{S}_\phi^\vee & \xrightarrow{1:1} \Pi^{G^{\pm}}_\phi \\
	\chi & \mapsto \sigma_{\phi, \chi}.
\end{aligned}\end{equation}
As before, \eqref{eqn:LLC-SO} and \eqref{eqn:LLC-SO-chi} preserve temperedness and square-integrability, and the general case reduces to the tempered case by forming Langlands quotients.

Below is the main result of this article, which compares the correspondence $\mathrm{TW}^*$ (Definition \ref{def:TW-star}) and the LLC \eqref{eqn:LLC-Mp}, \eqref{eqn:LLC-SO}, \eqref{eqn:LLC-SO-chi}.

\begin{theorem}\label{prop:main}
	Let $\pi$ be an irreducible genuine representation of $\tilde{G}$ lying in $\mathcal{G}_{\bpsi}^{\pm}$, and put $\sigma = \mathrm{TW}^*(\pi)$. If $\pi = \pi_{\phi, \chi}$ and $\sigma = \sigma_{\phi^\circ, \chi^\circ}$ under the LLC for $\tilde{G}$ and $G^{\pm}$ respectively, then
	\[ \phi^\circ = \phi, \quad \chi^\circ = \chi\nu_\phi. \]
\end{theorem}
\begin{proof}
	Combine the upcoming Lemma \ref{prop:basic-case} (the basic case $n=1$ and $\pi$ square-integrable), Lemma \ref{prop:to-tempered} (reduction to $\pi$ tempered), Lemma \ref{prop:to-gp} (reduction to $\pi$ tempered and $\phi$ is of good parity), Lemma \ref{prop:to-discrete} (reduction to $\pi$ square-integrable) and Lemma \ref{prop:to-Jacquet} (reduction from the previous case to tempered case of smaller $n$).
\end{proof}

The proofs of the required Lemmas occupy the rest of this article.

The assertion $\phi^\circ = \phi$ of Theorem \ref{prop:main} is already proved in \cite[Corollary 21]{GS2} and \cite[Remark after Theorem 3.6]{TW18}. We give another proof in the tempered case below, which will be used in the proof of Theorem \ref{prop:main}.

\begin{lemma}\label{prop:phi-tempered}
	Let $\pi$ be a tempered irreducible genuine representation of $\tilde{G}$ lying in $\mathcal{G}_{\bpsi}^{\pm}$. Then the representation $\sigma = \mathrm{TW}^*(\pi)$ of $G^{\pm}(F)$ has the same L-parameter as $\pi$.
\end{lemma}
\begin{proof}
	To begin with, assume $\pi$ is square-integrable. Then so is $\sigma$. Denote by $\phi$ (resp.\ $\phi^\circ$) the L-parameter of $\pi$ (resp.\ $\sigma$). Let $\pi_{\GL}$ (resp.\ $\pi^\circ_{\GL}$) be the tempered representation of $\GL(2n, F)$ parameterized by $\phi$ (resp.\ $\phi^\circ$). Embed $G \times \GL(2n)$ (resp.\ $G^{\pm} \times \GL(2n)$) as a Levi factor of some parabolic $Q$ (resp.\ $Q^{\pm}$) in a larger symplectic group $S$ (resp.\ special odd orthogonal group $S^{\pm}$). Claim: $i^{\tilde{S}}_{\tilde{Q}}(\pi \boxtimes \pi_{\GL})$ is irreducible.
	
	Indeed, by \cite[Corollary 8.3]{GS1} this is equivalent via $\Theta$-lifting to the corresponding assertion for square-integrable representations of $G^\epsilon(F)$, where $\epsilon = \chi(-1)$ (with the $-1 \in \Sp(2n, \CC)$). The case for $G^\epsilon$ follows from Arthur's description of the Knapp--Stein $R$-group on the dual side. Detailed arguments can be found in \cite[Lemma 4.2.1]{Li24b} when $\epsilon = +$ (see also \cite[Theorem 11.1]{Xu17}); as for $\epsilon = -$, simply replace the reference in \textit{loc.\ cit.}\ to \cite{Ar13} by Ishimoto's work \cite{Is24}.
	
	By the claim and Theorem \ref{prop:compatibility-1}, $i^{S^{\pm}}_{Q^{\pm}}(\sigma \boxtimes \pi_{\GL})$ is irreducible as $Q \leftrightarrow Q^{\pm}$. The aforementioned result \cite[Lemma 4.2.1]{Li24b} and its variant for non-split $\SO$ imply that irreducibility occurs exactly when $\pi_{\GL} = \pi^\circ_{\GL}$. Hence $\phi = \phi^\circ$.
	
	Now assume $\pi$ is tempered and embed it into $i_{\tilde{P}}(\pi_0)$ for some $P = MU$ as above and some square-integrable $\pi_0$ on $\tilde{M}$. Theorem \ref{prop:compatibility-1} implies that $\sigma$ embeds into $i_{P^\pm}(\sigma_0)$ where $P \leftrightarrow P^{\pm}$ and $\sigma_0 = \mathrm{TW}^{\tilde{M}, *}(\pi_0)$. Again, standard properties of LLC assure that $\phi$ (resp.\ $\phi^\circ$) is the image of the L-parameter of $\pi_0$ (resp.\ $\sigma_0$); see \cite[Theorem 8.1]{GS1} for the metaplectic case. However, $\pi_0$ and $\sigma_0$ have the same L-parameter by the previous step.
\end{proof}

In all the statements above, one adopts Arthur's endoscopic classification of representations of $G^{\pm}(F)$ (see \cite{Ar13, Is20}) instead of Lusztig's (see \cite[Corollary 6.5]{Lu95}). These two parametrizations are expected to agree. We do not need such results in this article; nevertheless, the proof leads to the following weak result as a by-product.

\begin{proposition}\label{prop:phi-nr}
	Let $\pi$ be an irreducible genuine representation of $\tilde{G}$ lying in $\mathcal{G}_{\bpsi}^{\pm}$, then its L-parameter is trivial on $I_F \times \{1\} \subset \mathcal{L}_F$.
\end{proposition}

It is equivalent to the versions for $\mathcal{G}^{\pm}$. Due to the lack of an adequate reference for Arthur's parametrization, we shall give proof for $\mathcal{G}_{\bpsi}^{\pm}$ in \S\ref{sec:completion-proof}.

\section{The basic case}\label{sec:basic}
Recall that $\tilde{G} = \Mp(W)$ is the metaplectic group of rank $n := \frac{1}{2} \dim W$. In this section, we treat the basic case of Theorem \ref{prop:main} with $n=1$ and $\pi \in \Pi_{2, -}(\tilde{G})$ (i.e.\ square-integrable).

\begin{lemma}\label{prop:basic-case}
	Assume $n=1$. Then the statement $(\phi^\circ, \chi^\circ) = (\phi, \chi\nu_\phi)$ in Theorem \ref{prop:main} holds for all $\pi \in \Pi_{2, -}(\tilde{G})$.
\end{lemma}
\begin{proof}
	Since $\phi \in \Phi_{2, \mathrm{bdd}}(\tilde{G})$, it is a simple $2$-dimensional self-dual representation of $\mathcal{L}_F$ of symplectic type, hence $|I| = |I^+| = 1$, $\EuScript{S}_\phi = S_\phi = \bmu_2$ and
	\[ \nu_\phi(-1) = \epsilon\left(\frac{1}{2}, \phi, \bpsi \right). \]
	
	By Lemma \ref{prop:phi-tempered}, we have $\phi^\circ = \phi$. Vogan's LLC for $G^{\pm}$ has the property that $\chi^\circ$ is the trivial (resp.\ non-trivial) element of $\EuScript{S}_\phi^\vee$ in the $+$ (resp.\ $-$) case. See eg.\ \cite[\S 10]{GGP1}.
	
	On the side of $\tilde{G}$, let us begin with the case of general $n$. Recall that since we are working with eightfold coverings, there is a canonical preimage of $-1 \in G(F)$ in $\tilde{G}$, still denoted by $-1$ here, which satisfies $(-1)^2 = 1$ and $\omega^{\pm}_{\bpsi}(-1) = \pm\identity$. In fact, under the embedding
	\[ (F^\times)^n \times \bmu_8 \simeq T(F) \times \bmu_8 \simeq \tilde{T} \hookrightarrow \tilde{G}, \]
	the $-1 \in \tilde{G}$ is the image of $(\underbracket{-1, \ldots, -1}_{n\;\text{terms}}, 1)$, as one sees from a computation in the Schrödinger model. We refer to \cite[Définition 2.8]{Li11} and \cite[Chapitre 2, II.6]{MVW87} for details.
	
	We now specialize to $n=1$. By \cite[Theorem 1.4 (ii)]{GS1} (due to Waldspurger for $n=1$), the value of $\chi \nu_\phi$ at $-1 \in \EuScript{S}_\phi$ equals $\omega_\pi(-1)$, known as the central sign of $\pi$.
	
	We use the description of blocks $\mathcal{G}_{\bpsi}^{\pm}$. In the $+$ case, $\pi$ embeds into $i_{\tilde{B}^{\leftarrow}}(\eta)$ for some unramified character $\eta: F^{\times} \to \mathbb{C}^{\times}$ of $T(F)$, inflated to a genuine character of $\tilde{T} \simeq T(F) \times \bmu_8$. Hence $\omega_\pi(-1) = \eta(-1) = 1$.
	
	In the $-$ case, the only possibility is $\pi = \omega_{\bpsi}^-$. We have seen that $\omega_{\bpsi}^-(-1) = -\mathrm{id}$, hence $\omega_\pi(-1) = -1$.
	
	All in all, we obtain $\chi^\circ = \chi \nu_\phi$.
\end{proof}

\begin{remark}\label{rem:basic-L}
	As $n=1$, elements in $\Pi_{2, -}(\tilde{G})$ can be explicitly classified as follows (see for example \cite[\S 2.17]{GanP} for details):
	\begin{itemize}
		\item Steinberg representations $\mathrm{st}_{\bpsi, \chi}$, which is the socle of $i_{\tilde{B}^{\leftarrow}}(\chi|\cdot|^{1/2})$ (note that $B^{\leftarrow} = B^{\rightarrow}$ when $n=1$);
		\item odd Weil representations $\omega_{\bpsi_a}^-$, which are all supercuspidal;
		\item the other genuine supercuspidal representations.
	\end{itemize}
	Here $\chi$ denotes a quadratic character of $F^{\times}$ and $a \in F^{\times}$. The Weil representations $\omega_{\bpsi_a}^{\pm}$ are defined using $\bpsi_a$, but their isomorphism classes depend only on the coset $a F^{\times 2}$.
	
	Among them, those $\pi$ lying in $\mathcal{G}_{\bpsi}^+$ (resp.\ $\mathcal{G}_{\bpsi}^-$) are $\mathrm{st}_{\bpsi, \chi}$ with $\chi$ unramified (resp.\ just $\omega_{\bpsi}^-$). To see this, simply describe $\mathcal{G}_{\bpsi}^{\pm}$ in terms of cuspidal supports (see \S\ref{sec:blocks}), and use the fact \cite[\S 2.7]{GanP} that $\omega_{\bpsi_a}^- \simeq \omega_{\bpsi_b}^-$ if and only if $a/b \in F^{\times 2}$.
	
	Recall that the LLC for $\tilde{G}$ is defined through $\Theta$-lifting. Consulting the table in \cite[\S 2.17]{GanP}, the L-parameter of $\pi = \pi_{\phi, \chi}$ in $\mathcal{G}_{\bpsi}^{\pm}$ is seen to be $\phi = \xi \boxtimes r(2)$ where $\xi = \chi$ (resp.\ $\xi = \mathbf{1}$) when $\pi = \mathrm{st}_{\bpsi, \chi}$ (resp.\ $\pi = \omega_{\bpsi}^-$), and $\chi \in \EuScript{S}_\phi^\vee$ is determined by $\chi(-1) = 1$ (resp.\ $-1$) when $\pi$ is $\Theta$-lifted from $\SO(V^+)$ (resp.\ $\SO(V^-)$).
	
	In this way, one can also prove Lemma \ref{prop:basic-case} by explicit verification.
\end{remark}

\section{Reduction to tempered case}\label{sec:reduction-to-tempered}
For general $n$, we shall reduce the main Theorem \ref{prop:main} to the tempered case via Langlands quotients, and then to the case of good parity.

\begin{lemma}\label{prop:to-tempered}
	If the statement $(\phi^\circ, \chi^\circ) = (\phi, \chi\nu_\phi)$ in Theorem \ref{prop:main} holds for all tempered irreducible genuine representations of metaplectic groups of rank $\leq n$, then it holds for $\tilde{G}$ in general. Here the representations are assumed to belong to $\mathcal{G}_{\bpsi}^{\pm}$ or its avatars.
\end{lemma}
\begin{proof}
	Let $\pi$ and $\sigma$ be as in the statement of Theorem \ref{prop:main}. In view of Theorem \ref{prop:compatibility-1}, we may take a parabolic subgroup $P \supset B^{\leftarrow}$ (the $+$ case) or $P \supset P^1$ (the $-$ case) of $G$, with Levi decomposition $P = MU$ where $M \supset T$, and similarly a parabolic subgroup $P^{\pm} \supset P_{\min}^{\pm}$ of $G^\pm$, with Levi decomposition $P^{\pm} = M^{\pm} U^{\pm}$ where $M^{\pm} \supset M_{\min}^{\pm}$, such that
	\begin{itemize}
		\item $P \leftrightarrow P^{\pm}$, $M \leftrightarrow M^{\pm}$;
		\item there exists $\pi_{\tilde{M}}$ (resp.\ $\sigma_{\tilde{M}}$) such that $\pi$ (resp.\ $\sigma$) is the Langlands quotient of $i_{\tilde{P}}(\pi_{\tilde{M}})$ (resp.\ $i_{P^\pm}(\sigma_{M^\pm})$).
	\end{itemize}

	Write $M = \Sp(W^\flat) \times \prod_{k=1}^r \GL(n_k)$ and $M^{\pm} = \SO(V^{\pm, \flat}) \times \prod_{k=1}^r \GL(n_k)$ where $\dim W^{\flat} = \dim V^{\pm, \flat} + 1 = 2n - 2\sum_{k=1}^r n_k$. Write $\pi_{\tilde{M}} = \pi_0 \boxtimes \pi_1$ (resp.\ $\sigma_{M^{\pm}} = \sigma_0 \boxtimes \sigma_1$) so that
	\begin{itemize}
		\item $\pi_1 = \sigma_1$ as representations of $\prod_{k=1}^r \GL(n_k, F)$,
		\item $\sigma_0 = \mathrm{TW}^{\flat, *}(\pi_0)$ as representations of $\SO(V^{\pm, \flat})$, both are tempered.
	\end{itemize}

	Let $\phi_1$ be the L-parameter for $\pi_1 = \sigma_1$, a representation of $\mathcal{L}_F$ of dimension $\sum_k n_k$. By assumption, if $\pi_0 = \pi_{\phi_0, \chi_0}$ then $\sigma_0 = \sigma_{\phi_0, \chi_0 \nu_{\phi_0}}$ where $\nu_{\phi_0} \in \EuScript{S}_{\phi_0}^\vee$ is as in Definition \ref{def:nu-phi}. It follows that $\pi$ and $\sigma$ have the same L-parameter
	\[ \phi = \phi_0 \oplus \left( \phi_1 \oplus \check{\phi}_1 \right). \]
	
	The compatibility of LLC with Langlands quotients (see \S\ref{sec:LLC}) implies that $\chi$ (resp.\ $\chi^\circ$) is the image of $\chi_0$ (resp.\ $\chi_0 \nu_{\phi_0}$) under the dual of $\EuScript{S}_{\phi_0} \simeq \EuScript{S}_\phi$.

	Since $\pi_{\tilde{M}}$ is a $P$-positive unramified twist of a tempered representation, all summands indexed by $I^+$ in the decomposition of $\phi$ come from $\phi_0$. In view of Definition \ref{def:nu-phi}, this implies the image of $\nu_{\phi_0}$ under $\EuScript{S}_{\phi_0}^\vee \simeq \EuScript{S}_\phi^\vee$ is exactly $\nu_\phi$. The proof is complete.
\end{proof}

The next step follows a similar paradigm.

\begin{definition}\label{def:gp}
	Let $\phi \in \Phi_{\mathrm{bdd}}(\tilde{G})$ and write
	\[ \phi = \bigoplus m_i \phi_i, \quad I = I^+ \sqcup I^- \sqcup J \sqcup J' \]
	as in \S\ref{sec:L-parameters}. If $I = I^+$, we say $\phi$ is of \emph{good parity}.
\end{definition}

\begin{lemma}\label{prop:to-gp}
	If the statement $(\phi^\circ, \chi^\circ) = (\phi, \chi\nu_\phi)$ in Theorem \ref{prop:main} holds for all tempered irreducible genuine representations of metaplectic groups of rank $\leq n$ whenever $\phi$ is of good parity, then it holds for all tempered irreducible genuine representations $\pi$ of $\tilde{G}$. Here the representations are assumed to belong to $\mathcal{G}_{\bpsi}^{\pm}$ or its avatars.
\end{lemma}
\begin{proof}
	Similar to the proof of Lemma \ref{prop:to-tempered}. By Lemma \ref{prop:phi-tempered} we have $\phi = \phi^\circ$. Write
	\begin{gather*}
		\phi = \phi_{\mathrm{gp}} \oplus \left( \phi_{\mathrm{ngp}} \oplus \phi_{\mathrm{ngp}}^\vee \right), \\
		\phi_{\mathrm{gp}} := \bigoplus_{i \in I^+} m_i \phi_i, \quad \phi_{\mathrm{ngp}} := \bigoplus_{i \in I^-} \frac{m_i}{2} \phi_i \oplus \bigoplus_{j \in J} m_j \phi_j.
	\end{gather*}
	
	Then one may take $P = MU \supset B^{\leftarrow}$ (resp.\ $\supset P^1$) in the $+$ (resp.\ $-$) case such that $\phi_{\tilde{M}} := (\phi_{\mathrm{gp}}, \phi_{\mathrm{ngp}})$ lies in $\Phi_{\mathrm{bdd}}(\tilde{M})$ and $\phi_{\tilde{M}} \mapsto \phi$ via \eqref{eqn:Phi-M-G}. The metaplectic component $\phi_{\mathrm{gp}}$ of $\phi_{\tilde{M}}$ is then of good parity. Take the subgroups $P^{\pm}$ and $M^{\pm}$ of $G^{\pm}$ that correspond to $P$ and $M$, respectively.

	By \eqref{eqn:S} and Definition \ref{def:nu-phi}, we have
	\[ \EuScript{S}_{\phi_{\mathrm{gp}}} = \EuScript{S}_{\phi_{\tilde{M}}} \rightiso \EuScript{S}_\phi. \]
	View $\chi$, $\chi^\circ$ and $\nu_\phi$ as elements of $\EuScript{S}_{\phi_{\mathrm{gp}}}$ under the dual of these canonical isomorphisms. Then $\nu_\phi$ coincides with $\nu_{\phi_{\mathrm{gp}}}$.
	
	By Arthur's theory, $\sigma$ is induced irreducibly from the representation $\sigma_{\tilde{M}}$ indexed by $(\phi_{\tilde{M}}, \chi^\circ)$. Indeed, it suffices to inspect the Knapp--Stein $R$-group in terms of L-parameters to obtain irreducibility; this is part of the LIR. Take $\pi_{\tilde{M}}$ such that $\mathrm{TW}^{\tilde{M}, *}(\pi_{\tilde{M}}) = \sigma_{\tilde{M}}$, then $\pi$ is induced irreducibly from $\pi_{\tilde{M}}$ by Theorem \ref{prop:compatibility-1}.
	
	By assumption, $\pi_{\tilde{M}}$ is indexed by $(\phi_{\tilde{M}}, \chi^\circ \nu_\phi)$. It follows that $\chi = \chi^\circ \nu_\phi$ by the relation \cite[Theorem 8.1]{GS1} between LLC for $\tilde{G}$ and parabolic induction, as desired.
\end{proof}

\section{Reduction to square-integrable case}\label{sec:reduction-to-L2}
Assume $n > 1$ in this section. Fix $\phi = \bigoplus_{i \in I} m_i \phi_i \in \Phi_{\mathrm{bdd}}(\tilde{G})$ and decompose
\[ I^+ = I^+_{\mathrm{even}} \sqcup I^+_{\mathrm{odd}} \]
according to the parity of $m_i$ for each $i \in I^+$; see \S\ref{sec:L-parameters}.

\subsection{On certain centralizers}\label{sec:LIR-parameters}
There exists $(P, \phi_{\tilde{M}})$ where $P \supset B^{\leftarrow}$ (resp.\ $\supset P^1$) in the $+$ (resp.\ $-$) case is a parabolic subgroup of $G$ with Levi decomposition $P = MU$, $M \supset T$, and $\phi_{\tilde{M}} \in \Phi_{2, \mathrm{bdd}}(\tilde{M})$ satisfies $\phi_{\tilde{M}} \mapsto \phi$ via \eqref{eqn:Phi-M-G}. The datum $P$ is unique. Identify the Weyl group $\Omega^G_0$ (resp.\ $\Omega^M_0$) of $(G, T)$ (resp.\ $(M, T)$) with the Weyl group of $\tilde{G}^\vee$ (resp.\ $\tilde{M}^\vee$). Define
\[ \Omega^G(M) := \left\{ w \in \Omega^G_0: wMw^{-1} = M \right\} \big/ \Omega^M_0. \]
Regard $\tilde{M}^\vee$ as a standard Levi subgroup of $\tilde{G}^\vee$ and define $\Omega^{\tilde{G}^\vee}(\tilde{M}^\vee)$ analogously. There is a canonical isomorphism $\Omega^{\tilde{G}^\vee}(\tilde{M}^\vee) \simeq \Omega^G(\tilde{M})$, and $\phi_{\tilde{M}}$ is unique modulo $\Omega^G(M)$.

Write $M$ as $\Sp(W^\flat) \times \prod_{k=1}^r \GL(n_k)$ and $\phi_{\tilde{M}} = (\phi_0, \phi_1)$ where $\phi_0$ (resp.\ $\phi_1$) is the metaplectic (resp.\ $\GL$) part of $\phi_{\tilde{M}}$. Then
\[ \phi_0 = \bigoplus_{i \in I^+_{\mathrm{odd}}} \phi_i. \]

There is a natural inclusion
\[ \EuScript{S}_{\phi_0} = \EuScript{S}_{\phi_{\tilde{M}}} \simeq \bmu_2^{I^+_{\mathrm{odd}}} \hookrightarrow \bmu_2^{I^+} = \EuScript{S}_\phi \]
admitting an evident retraction: $\EuScript{S}_\phi = \EuScript{S}_{\phi_0} \times R_\phi$ where $R_\phi = \bmu_2^{I^+_{\mathrm{even}}}$. Definition \ref{def:nu-phi} implies
\begin{equation}\label{eqn:LIR-nu}
	\nu_\phi|_{\EuScript{S}_{\phi_0}} = \nu_{\phi_0}.
\end{equation}

Let $A_{\tilde{M}^\vee}$ be the maximal central torus in $\tilde{M}^\vee$. Define as in \cite[\S 3.2]{Is20} the finite groups
\[\begin{tikzcd}[row sep=small, column sep=small]
	& S_\phi^\natural(\tilde{M}, \tilde{G}) \arrow[phantom, r, "{:=}" description] & N_{S_\phi}(A_{\tilde{M}^\vee}) \big/ N_{S_\phi^\circ}(A_{\tilde{M}^\vee}) \\
	& \mathfrak{N}_\phi(\tilde{M}, \tilde{G}) \arrow[twoheadrightarrow, d] \arrow[twoheadrightarrow, u] \arrow[phantom, r, "{:=}" description] & N_{S_\phi}(A_{\tilde{M}^\vee}) \big/ Z_{S_\phi^\circ}(A_{\tilde{M}^\vee}) \\
	\Omega^G(M) & \mathrm{W}_\phi(\tilde{M}, \tilde{G}) \arrow[phantom, r, "{:=}" description] \arrow[hookrightarrow, l] & N_{S_\phi}(A_{\tilde{M}^\vee}) \big/ Z_{S_\phi}(A_{\tilde{M}^\vee});
\end{tikzcd}\]
the downward surjection fits into a short exact sequence
\begin{equation}\label{eqn:W-N-ses}
	1 \to \left( \EuScript{S}_{\phi_0} = \EuScript{S}_{\phi_{\tilde{M}}} \right) \to \mathfrak{N}_\phi(\tilde{M}, \tilde{G}) \to \mathrm{W}_\phi(\tilde{M}, \tilde{G}) \to 1.
\end{equation}

Since $\phi_{\tilde{M}} \in \Phi_{2, \mathrm{bdd}}(\tilde{M})$, there are further simplifications to \textit{loc.\ cit.}: $A_{\tilde{M}^\vee}$ is a maximal torus in $S_\phi$, and
\[ S^\natural_\phi(\tilde{M}, \tilde{G}) = \EuScript{S}_\phi; \]
see the discussions after \cite[Theorem 2.4.1]{Ar13}. Comparing with \eqref{eqn:S}, we see that $A_{\tilde{M}^\vee}$ decomposes accordingly, and:
\begin{itemize}
	\item $\mathrm{W}_\phi(\tilde{M}, \tilde{G})$ is the direct product of the Weyl groups of $\Or(m_i, \CC)$, $\Sp(m_i, \CC)$ and $\GL(m_j, \CC)$, where the Weyl group of $\Or(m_i, \CC)$ means the group of signed permutations without any parity constraint, i.e.\ $(\Z/2\Z)^t \rtimes \mathfrak{S}_t$ where $t$ is the rank of $\Or(m_i, \CC)$;
	\item $\mathrm{W}_\phi(\tilde{M}, \tilde{G})$ acts on $X^*(A_{\tilde{M}^\vee})$ by signed permutations as described above;
	\item $\mathfrak{N}_\phi(\tilde{M}, \tilde{G})$ is the direct product of $\mathrm{W}_\phi(\tilde{M}, \tilde{G})$ with $\bmu_2^{I^+_{\mathrm{odd}}}$, accounting of the $-\identity \in \Or(m_i, \CC)$ for $i \in I^+_{\mathrm{odd}}$.
\end{itemize}

Hence \eqref{eqn:W-N-ses} splits canonically. Moreover, the surjection $\mathfrak{N}_\phi(\tilde{M}, \tilde{G}) \twoheadrightarrow \EuScript{S}_\phi$ restricts via \eqref{eqn:W-N-ses} to the inclusion $\EuScript{S}_{\phi_0} \hookrightarrow \EuScript{S}_\phi$. In fact, on each factor of  $\mathfrak{N}_\phi(\tilde{M}, \tilde{G})$ it has the following simple description:
\begin{itemize}
	\item on the factor indexed by $i \in I^+_{\mathrm{odd}}$, it is the projection onto the $i$-th component of the $\bmu_2^{I^+}$;
	\item on the factor indexed by $i \in I^+_{\mathrm{even}}$, it maps an element of the Weyl group of $\Or(m_i, \CC)$ to its ``total sign'', placed at the $i$-th component of $\bmu_2^{I^+}$;
	\item on the other summands, the homomorphism is trivial.
\end{itemize}

\begin{lemma}\label{prop:surjectivity}
	The composite of $\mathrm{W}_\phi(\tilde{M}, \tilde{G}) \hookrightarrow \mathfrak{N}_\phi(\tilde{M}, \tilde{G}) \twoheadrightarrow \EuScript{S}_\phi \twoheadrightarrow R_\phi$ is surjective, where the first arrow is the section to \eqref{eqn:W-N-ses}.
\end{lemma}
\begin{proof}
	Clear from the explicit description above and $R_\phi = \bmu_2^{I^+_{\mathrm{even}}}$. Cf.\ \cite[(2.4.3)]{Ar13}.
\end{proof}

\subsection{Normalized intertwining operators}
Keep the assumptions on $\phi \in \Phi_{\mathrm{bdd}}(\tilde{G})$ and realize $\phi$ as the image of $\phi_{\tilde{M}} \in \Phi_{2, \mathrm{bdd}}(\tilde{M})$. Suppose that $\pi = \pi_{\phi, \chi}$ in $\mathcal{G}_{\bpsi}^{\pm}$ is given. Let $\sigma := \mathrm{TW}^*(\pi) = \sigma_{\phi^\circ, \chi^\circ}$. Lemma \ref{prop:phi-tempered} implies $\phi^\circ = \phi$. Using $\EuScript{S}_{\phi_{\tilde{M}}} = \EuScript{S}_{\phi_0}$, we put
\[ \chi_0 := \chi|_{\EuScript{S}_{\phi_0}}, \quad \pi_{\tilde{M}} := \pi_{\phi_{\tilde{M}}, \chi_0}. \]
Then $\pi$ appears in $i_{\tilde{P}}(\pi_{\tilde{M}})$, and $\pi_{\tilde{M}}$ is the unique member of $\Pi_{\phi_{\tilde{M}}}$ with this property; moreover, $\pi$ appears with multiplicity $1$. Indeed, in view of \cite[Theorem 8.1]{GS1}, all these follow from the corresponding properties for $G^{\pm}$ --- see the proof of \cite[Proposition 2.4.3]{Ar13} for $G^+$ or the counterparts for $G^-$ in \cite{Is24}. Therefore, $\pi_{\tilde{M}}$ lies in $\mathcal{G}_\bpsi^{\tilde{M}, \pm}$.

We have the correspondence $P \leftrightarrow P^\pm$ for parabolic subgroups and $M \leftrightarrow M^\pm$ for their Levi factors. The same constructions for $G^{\pm}$ yield
\[ \chi^\circ_0 := \chi^\circ|_{\EuScript{S}_{\phi_0}}, \quad \sigma_{\tilde{M}} := \sigma_{\phi_{\tilde{M}}, \chi^\circ_0} \]
so that $\sigma$ appears in $i_{P^{\pm}}(\sigma_{\tilde{M}})$ with multiplicity $1$, and $\sigma_{\tilde{M}}$ lies in $\mathcal{G}^{M^{\pm}}$.

There is a canonical isomorphism $\Omega^G(M) \simeq \Omega^{G^\pm}(M^\pm)$.

\begin{lemma}
	Modulo $\Omega^G(M)$, we may arrange that $\sigma_{\tilde{M}} = \mathrm{TW}^{\tilde{M}, *}(\pi_{\tilde{M}})$.
\end{lemma}
\begin{proof}
	Since $\mathrm{TW}$ is an isomorphism of Hilbert algebras, one can compare the classifications of tempered representations in terms of square-integrable ones (i.e.\ discrete series) for $\tilde{G}$ and $G^{\pm}$ (part of Knapp--Stein theory, which features the same ambiguity from $\Omega^G(M)$) via $\mathrm{TW}$, by using Theorem \ref{prop:compatibility-1}.
\end{proof}

Assume henceforth $M \neq G$. We identify an element $w \in \Omega^G(M)$ with the unique representative of minimal length in the $\Omega^M_0$-coset. Here the length is computed with respect to the $B^{\leftarrow}$-simple roots.

\begin{definition}\label{def:R-operator}
	Given $w \in \Omega^G(M)$ with ${}^{\tilde{w}}\pi_{\tilde{M}} \simeq \pi_{\tilde{M}}$ where $\tilde{w}$ is chosen as in \S\ref{sec:tw}, define the corresponding normalized intertwining operator
	\begin{equation*}
		R(w) := r_{\tilde{P}^w | \tilde{P}}(\pi_{\tilde{M}})^{-1} i_{\tilde{P}}(A) J_w(\pi_{\tilde{M}}) \in \End_{\tilde{G}}\left( i_{\tilde{P}}(\pi_{\tilde{M}}) \right)
	\end{equation*}
	as in \cite[Definition 10.4.1]{CL23}, by using:
	\begin{itemize}
		\item the standard intertwining operator $J_w(\pi_{\tilde{M}})$ in Definition \ref{prop:compatibility-2};
		\item an isomorphism $A: {}^{\tilde{w}} \pi_{\tilde{M}} \to \pi_{\tilde{M}}$ between unitary representations;
		\item the normalizing factor $r_{\tilde{P}^w | \tilde{P}}(\pi_{\tilde{M}})$ where $\tilde{P}^w := w^{-1} \tilde{P} w$.
	\end{itemize}
	More rigorously, one should allow unramified twists $\pi_{\tilde{M}} \otimes \alpha$ in $J_w(\cdot)$ and $r_{\tilde{P}^w | \tilde{P}}(\cdot)$, and view $R(w)$ as the value at $\alpha = \mathbf{1}$ in a meromorphic family, where $\alpha$ is an unramified character of $\prod_{k=1}^r \GL(n_k, F)$.
\end{definition}

Up to a positive constant depending only on $w$ and $M$, the normalizing factor $r_{\tilde{P}^w | \tilde{P}}(\pi_{\tilde{M}})$ in Definition \ref{def:R-operator} is the same as the one prescribed in \cite[\S 7.3]{Is20}, which can be given in terms of $\gamma$-factor.

\begin{definition}\label{def:Rpm-operator}
	The same constructions apply to $\sigma_{\tilde{M}}$ and the same $w$, and yield the normalized intertwining operators
	\[ R^{\pm}(w) := r_{P^{\pm, w}|P^{\pm}}(\sigma_{\tilde{M}})^{-1} i_{P^{\pm}}(A^{\pm}) J^{\pm}_w(\sigma_{\tilde{M}}) \in \End_{G^{\pm}(F)}\left( i_{P^\pm}(\sigma_{\tilde{M}}) \right) \]
	(more rigorously, as the value at $\alpha = \mathbf{1}$ of meromorphic families in $\sigma_{\tilde{M}} \otimes \alpha$), in the following manner.
	\begin{itemize}
		\item Following \cite{Is20}, the Langlands--Shelstad representative $\dot{w}$ of $w$ in $G^\pm(F)$ is used in the standard intertwining operator $J_w^{\pm}(\sigma_{\tilde{M}})$ (see \cite[\S 2.3]{KMSW} for the $-$ case). This representative was denoted by $\ddot{w}$ in the $-$ case in \cite{CL23}; we relinquish this notation.
	
		\item The isomorphism $A^{\pm}: {}^{\dot{w}} \sigma_{\tilde{M}} \to \sigma_{\tilde{M}}$ is chosen to be compatible with $A$ via $\mathrm{TW}^{\tilde{M}, *}$. See \cite[Corollaries 10.2.4 and 10.2.6]{CL23} for the explicit conditions.
		
		\item The normalizing factor $r_{P^{\pm, w}|P^{\pm}}(\sigma_{\tilde{M}})$ is given by the same $\gamma$-factor when the Haar measures are suitably chosen, since $\phi^\circ = \phi$.
	\end{itemize}
\end{definition}

Since $\mathrm{TW}^{\tilde{M}}$ is an isomorphism of Hilbert algebras, $A^{\pm}$ is unitary as $A$ is. The resulting $R(w)$ and $R^{\pm}(w)$ are both unitary operators. A definitive choice of $(A, A^{\pm})$ will be made in \S\ref{sec:application-LIR}.

\begin{definition}
	Suppose that $a$ and $b$ are either complex numbers or endomorphisms on the same complex vector space. If there exists $t \in \mathbb{R}_{> 0}$ such that $a = tb$, then we write $a \sim b$.
\end{definition}

Modulo $\sim$ allows us to forget about Haar measures, formal degrees, real powers of $|2|_F$, etc.

\begin{lemma}\label{prop:R-ratio}
	View $R(w)$ (resp.\ $R^{\pm}(w)$) as a morphism in $\mathcal{G}_{\bpsi}^{\pm}$ (resp.\ $\mathcal{G}^{\pm}$), thus as an endomorphism of the relevant Hecke module. Under the identifications furnished by $\mathrm{TW}^*$, we have
	\[ R(w) \sim \begin{cases}
		R^+(w), & \text{in the $+$ case} \\
		(-1)^{t(w)} R^-(w), & \text{in the $-$ case.}
	\end{cases}\]
	Here $t(w) \in \mathbb{Z}_{\geq 0}$ is as in Theorem \ref{prop:compatibility-2}.
\end{lemma}
\begin{proof}
	This is the content of \cite[\S\S 10.2--10.3]{CL23} after taking $\sim$ classes. Note that we are using the same representatives of $w$ and the same Haar measures to define $R(w)$ in both the $\pm$ cases, in order to maintain compatibility with \cite{Is20}; this gives an extra factor $(-q^{-1})^{t(w)}$ when comparing $R(w)$ and $R^-(w)$, where $q$ is the residual cardinality of $F$; see \cite[Theorem 9.2.3, Remark 10.2.7]{CL23}.
\end{proof}

\subsection{Application of LIR}\label{sec:application-LIR}
Hereafter, we assume $\phi \in \Phi_{\mathrm{bdd}}(\tilde{G})$ is of good parity (Definition \ref{def:gp}). The aim is to show $\chi^\circ = \chi \nu_\phi$.

Since $M \subsetneq G$, the restrictions of $\chi^\circ$ and $\chi$ to $\EuScript{S}_{\phi_0}$ are known to be related by:
\begin{equation}\label{eqn:LIR-induction-chi}
	\chi^\circ_0 = \chi_0 \nu_{\phi_0}.
\end{equation}

Let $p$ be the composite of the maps $\mathrm{W}_\phi(\tilde{M}, \tilde{G}) \hookrightarrow \mathfrak{N}_\phi(\tilde{M}, \tilde{G}) \twoheadrightarrow \EuScript{S}_\phi$ in Lemma \ref{prop:surjectivity}. To attain our goal, by \eqref{eqn:LIR-nu}, \eqref{eqn:LIR-induction-chi} and Lemma \ref{prop:surjectivity} it suffices to show
\begin{equation}\label{eqn:to-discrete}
	\chi^\circ|_{p\mathrm{W}_\phi(\tilde{M}, \tilde{G})} = \left( \chi \nu_\phi \right)|_{p\mathrm{W}_\phi(\tilde{M}, \tilde{G})}.
\end{equation}

\begin{lemma}\label{prop:to-discrete}
	The equality \eqref{eqn:to-discrete} holds. Consequently, the statement $(\phi^\circ, \chi^\circ) = (\phi, \chi\nu_\phi)$ in Theorem \ref{prop:main} holds for $\pi$ by granting its validity for all square-integrable irreducible genuine representations of metaplectic groups of rank $< n$ lying in the avatars of $\mathcal{G}_{\bpsi}^{\pm}$.
\end{lemma}
\begin{proof}
	Combining \eqref{eqn:LIR-induction-chi} and the description of $\mathfrak{N}_\phi(\tilde{M}, \tilde{G}) \twoheadrightarrow \EuScript{S}_\phi$ in \S\ref{sec:LIR-parameters}, it suffices to consider $p(w)$ for $w$ of the form
	\begin{align*}
		(1, 0, \ldots, 0) \rtimes 1 & \in (\Z/2\Z)^{m_i/2} \rtimes \mathfrak{S}_{m_i/2} \\
		& \simeq \text{the Weyl group of}\; \Or(m_i, \CC)
	\end{align*}
	for all $i \in I^+_{\mathrm{even}}$, since their images generate $R_\phi$. Fix $i \in I^+_{\mathrm{even}}$ and $w$ as above, and put $x_w := p(w) \in \EuScript{S}_\phi$.

	Recall that $\mathrm{W}_\phi(\tilde{M}, \tilde{G}) \hookrightarrow \Omega^G(M)$. In \cite[\S 7.3]{Is20}, Ishimoto defined normalized intertwining operators for $\tilde{G}$ and $G^\pm(F)$ (here evaluated at $\mathbf{s} = 0$), abbreviated here as $\mathcal{R}(w)$ and $\mathcal{R}^{\pm}(w)$ respectively, whose definitions will be reviewed later. They are proportional to $R(w)$ and $R^{\pm}(w)$.
	
	Since $R(w)$ and $R^\pm(w)$ are unitary operators, we may denote by $\mathcal{R}(w) : R(w)$ the unique complex number $z$ such that $z R(w) = \mathcal{R}(w)$ and $|z| = 1$; ditto for $\mathcal{R}^\pm(w) : R^\pm(w)$.
	
	The LIR for $\tilde{G}$ and $G^\pm$ \cite[Theorem 4.2, Hypothesis 5.2]{Is20} asserts that $\mathcal{R}(w)$ and $\mathcal{R}^{\pm}(w)$ act on $\pi$ and $\pi^{\pm}$ by scalars $\chi(x_w)$ and $\chi^\circ(x_w)$, respectively. To obtain \eqref{eqn:to-discrete}, we must show that they differ by $\nu_\phi(x_w)$, and it suffices to show this up to $\sim$. We contend that
	\begin{equation}\label{eqn:R-ratio-sim}
		\mathcal{R}(w) : R(w) \sim
		\begin{cases}
			\nu_\phi(x_w) \cdot \mathcal{R}^{\pm}(w) : R^{\pm}(w), & \text{in the $+$ case}, \\
			(-1)^{t(w)} \nu_\phi(x_w) \cdot \mathcal{R}^{\pm}(w) : R^{\pm}(w), & \text{in the $-$ case}.
		\end{cases}
	\end{equation}
	
	The relation \eqref{eqn:R-ratio-sim} entails the desired property, since Lemma \ref{prop:R-ratio} says that in the $+$ (resp.\ $-$) case, the actions of $R(w)$ and $R^+(w)$ (resp.\ $R^-(w)$) on $\pi$ and $\sigma$ coincide (resp.\ differ by $(-1)^{t(w)}$) via $\mathrm{TW}^*$ modulo $\sim$.
	
	To analyze the ratios in \eqref{eqn:R-ratio-sim}, we summarize the (potential) differences between Definitions \ref{def:R-operator}, \ref{def:Rpm-operator} and Ishimoto's definitions in \cite{Is20} below.
	\begin{itemize}
		\item According to \cite[Lemmas 10.2.1--10.2.2]{CL23}:
		\begin{itemize}
			\item conjugation by $\tilde{w}$ acts on the Hecke module of $\pi_{\tilde{M}}$ through a canonical automorphism $\xi$ of the factor $\bigotimes_{i=1}^r H_k$ in \eqref{eqn:H-decomp}.
			\item ditto for $\dot{w}$ and $\sigma_{\tilde{M}}$ on $M^{\pm}(F)$, with the same automorphism $\xi$.
		\end{itemize}
		
		The operators $\mathcal{R}(w)$ and $\mathcal{R}^\pm(w)$ in \textit{loc.\ cit.}\ involve isomorphisms
		\[ \mathcal{A}: {}^{\tilde{w}}\pi_{\tilde{M}} \to \pi_{\tilde{M}}, \quad \mathcal{A}^{\pm}: {}^{\dot{w}} \sigma_{\tilde{M}} \to \sigma_{\tilde{M}}. \]
		It suffices to describe them on the level of Hecke modules. They are Whittaker-normalized on $\mathrm{GL}$-factors, and identity on the metaplectic or orthogonal factor, respectively. In particular, they are both unitary operators that match under $\mathrm{TW}^{\tilde{M}, *}$. Thus one can take $(A, A^\pm) := (\mathcal{A}, \mathcal{A}^{\pm})$ in the definition of $R(w)$ and $R^{\pm}(w)$.

		\item Define the representation $y(w, \phi) = y(w, \phi_{\mathbf{s} = 0})$ of $\mathcal{L}_F$ as in \cite[p.1575]{Is20}. Unraveling the definition of $\mathcal{R}^\pm(w)$ there, we see
		\[ \mathcal{R}^{\pm}(w) : R^{\pm}(w) \sim \epsilon(V^{\pm})^{\dim y(w, \phi)} = (\pm 1)^{\dim y(w, \phi)}; \]
		in the $-$ case it also equals $(-1)^{t(w)}$. Indeed, $t(w) = \dim y(w, \phi)$ by comparing \textit{loc.\ cit.}\ with \S\ref{sec:tw}.
	\end{itemize}
	
	Compared with Definition \ref{def:R-operator}, Ishimoto's $\mathcal{R}(w)$ contains the extra scalar
	\[ |2|_F^{2 y(w, \mathbf{s})} \gamma_F(\bpsi)^{\dim y(w, \phi)} \gamma\left( \frac{1}{2}, y(w, \phi_{\mathbf{s}}), \bpsi \right)^{-1} \]
	evaluated at $\mathbf{s} = 0$; the choice of representative of $w$ also differs.
	
	\begin{itemize}
		\item The factor $|2|_F^{2 y(w, \mathbf{s})}$ in $\mathcal{R}(w)$ is $1$ at $\mathbf{s} = 0$.
		
		\item In the intertwining integral in $\mathcal{R}(w)$, the representative $\tilde{w}_{\mathrm{Ishi}} \in \tilde{G}^{(2)}$ reviewed in \S\ref{sec:tw} is used; note that ${}^{\tilde{w}_{\mathrm{Ishi}}} \pi = {}^{\tilde{w}} \pi$. We infer from Proposition \ref{prop:gamma-computation} that
		\begin{equation*}
			\tilde{w} = \gamma_F(\bpsi)^{-t(w)} \tilde{w}_{\mathrm{Ishi}} = \gamma_F(\bpsi)^{-\dim y(w, \phi)} \tilde{w}_{\mathrm{Ishi}}.
		\end{equation*}
	
		Since $\tilde{w}^{-1}$ and $\tilde{w}_{\mathrm{Ishi}}^{-1}$ appear in the intertwining integrals, whilst $\pi_{\tilde{M}}$ is genuine, the factor $\gamma_F(\bpsi)^{\dim y(w, \phi)}$ in $\mathcal{R}(w)$ disappears if $\tilde{w}_{\mathrm{Ishi}}$ is replaced by $\tilde{w}$.
		\item Summing up,
		\[ \mathcal{R}(w) : R(w) \sim \gamma\left( \frac{1}{2}, y(w, \phi), \bpsi \right)^{-1}. \]
	\end{itemize}
	
	Now \eqref{eqn:R-ratio-sim} is reduced to
	\[ \gamma\left( \frac{1}{2}, y(w, \phi), \bpsi \right)^{-1} = \nu_\phi(x_w). \]
	By the choice of $w$, the left hand side is $\gamma\left(\frac{1}{2}, \phi_i , \bpsi \right)^{-1}$ and the right hand side is $\epsilon\left( \frac{1}{2}, \phi_i, \bpsi \right) = \epsilon\left( \frac{1}{2}, \phi_i, \bpsi \right)^{-1}$ as $\phi_i$ is self-dual of symplectic type. By the relation between $\gamma$ and $\epsilon$ factors (see \cite[p.441]{GR10}) and the property (same reference) that for all bounded L-parameter $\phi_{\mathrm{bdd}}: \mathcal{L}_F \to \GL(N, \CC)$ we have
	\[ s > 0 \implies L\left(s, \phi_{\mathrm{bdd}} \right) \in \CC^{\times}, \]
	the equality follows.
\end{proof}

\begin{remark}
	If $\phi \in \Phi_{\mathrm{bdd}}(\tilde{G})$ is no longer of good parity and $w$ is allowed to be a general element of $\mathrm{W}_\phi(\tilde{M}, \tilde{G})$, the arguments above still work. Although this extension is not mandatory for the proof, we give a sketch below.
	
	Since $M \subsetneq G$, Proposition \ref{prop:phi-nr} is known for $\phi$: in fact, its proof in \S\ref{sec:completion-proof} goes by reduction to square-integrable case. Hence $\phi_i|_{I_F}$ is trivial for all $i \in I$. We have $\gamma\left(\frac{1}{2}, \phi_i, \bpsi \right) = \epsilon\left(\frac{1}{2}, \phi_i, \bpsi \right)$ as before. Suppose $i \in I \smallsetminus I^+$.
	\begin{itemize}
		\item If $\phi_i$ is not self-dual, its contribution in $\gamma\left(\frac{1}{2}, y(w, \phi), \bpsi \right)$ will occur in pairs $\{\phi_i, \check{\phi}_i\}$, hence trivial by $\epsilon\left(\frac{1}{2}, \phi_i, \bpsi \right) \epsilon\left(\frac{1}{2}, \check{\phi}_i, \bpsi \right) = (\det\phi_i)(-1) = 1$.
		\item If $\phi_i$ is self-dual of orthogonal type, then $\phi_i = \chi \boxtimes r(a)$ for some unramified quadratic character $\chi$ of $F^{\times}$ and an odd $a \in \Z_{\geq 1}$. Let $\Frob \in \Weil{F}$ be a geometric Frobenius. By \cite[\S 2.2]{GR10}:
		\[ \epsilon\left( \frac{1}{2}, \phi_i, \bpsi \right) = \epsilon\left(\frac{1}{2}, \chi, \bpsi \right)^a (-\chi(\Frob))^{a-1} = \epsilon\left(\frac{1}{2}, \chi, \bpsi \right). \]
		Furthermore, the explicit formula \cite[(3.2.6.1)]{Ta79} for unramified $\epsilon$-factors and the fact that $\bpsi$ has conductor $4\mathfrak{o}$ imply $\epsilon\left(\frac{1}{2}, \chi, \bpsi \right) = 1$.
	\end{itemize}
	Hence one can argue as before to prove \eqref{eqn:R-ratio-sim}, considering only the contribution from $I^+$.
\end{remark}

\section{Reduction via Jacquet modules}\label{sec:Jacquet}
\subsection{Jacquet modules and transfer}\label{sec:Jacquet-transfer}
Let $\tilde{G} = \Mp(W)$ with $\dim W = 2n$ as before. Consider
\begin{itemize}
	\item a unitary irreducible supercuspidal representation $\rho$ of $\mathrm{GL}(d_\rho, F)$ where $1 \leq d_\rho \leq n$,
	\item $x \in \CC$ and the corresponding character $|\cdot|^x := |\det|^x$ of $\mathrm{GL}(d_\rho, F)$.
\end{itemize}

Let $P \supset B^{\leftarrow}$ (resp.\ $\supset P^1$) in the $+$ (resp.\ $-$) case be a parabolic subgroup of $G$ with Levi factor
\begin{gather*}
	M = G_- \times \GL(d_\rho) \supset T, \\
	G_- = \Sp(W_-), \quad W_- \subset W: \text{symplectic subspace}, \quad \dim W_- = 2n - 2d_\rho.
\end{gather*}

Let $\mathrm{Groth}(\tilde{G})$ be the $\mathrm{K}_0$-group of the abelian subcategory of $\tilde{G}\dcate{Mod}$ consisting of objects of finite length; it can also be identified with the $\Z$-module generated by irreducible genuine characters on $\tilde{G}$. Ditto for $\mathrm{Groth}(\tilde{G}_-)$.

Consider the functor of normalized Jacquet modules $r_{\tilde{P}}: \tilde{G}\dcate{Mod} \to \tilde{M}\dcate{Mod}$.

\begin{definition}\label{def:r-rho}
	Given $\rho$ and $x$ as above, for all $\pi \in \Pi_-(\tilde{G})$ we define $r^{\tilde{G}}_{\rho, x}(\pi)$ by extracting the part with $\GL(d_\rho, F)$-factor equal to $\rho |\cdot|^x$ in the semi-simplification of $r_{\tilde{P}}(\pi)$. This yields a homomorphism
	\[ r^{\tilde{G}}_{\rho, x}: \mathrm{Groth}(\tilde{G}) \to \mathrm{Groth}(\tilde{G}_-). \]
\end{definition}

Surely, the foregoing constructions pertain to odd orthogonal groups as well.

Now consider the following scenario. Let $\mathbf{G}^! \in \Endo_{\elli}(\tilde{G})$ be parameterized by $(n', n'') \in \Z_{\geq 0}^2$, so that
\[ G^! = \mathrm{SO}(2n'+1) \times \mathrm{SO}(2n''+1). \]
Let $\phi \in \Phi_{2, \mathrm{bdd}}(\tilde{G})$ and take the
\begin{itemize}
	\item $\phi^! = (\phi', \phi'') \in \Phi_{2, \mathrm{bdd}}(G^!)$ such that $\phi^! \mapsto \phi$ via the map $\Phi(G^!) \to \Phi(\tilde{G})$ induced by $(G^!)^\vee \hookrightarrow \tilde{G}^\vee$,
	\item the stable tempered character $S\Theta^{G^!}_{\phi^!}$ on $G^!(F)$, and
	\item its transfer $\trans_{\mathbf{G}^!, \tilde{G}}(S\Theta^{G^!}_{\phi^!})$ to $\tilde{G}$.
\end{itemize}
Therefore $\mathrm{Jord}(\phi) = \mathrm{Jord}(\phi') \sqcup \mathrm{Jord}(\phi'')$ (Definition \ref{def:Jord}).

Consider $(\rho, a) \in \mathrm{Jord}(\phi)$ and take $M := G_- \times \mathrm{GL}(d_\rho) \subset G$ as before. Take the Levi subgroup $M^! \subset G^!$, unique up to conjugacy, such that
\begin{itemize}
	\item $M^! = \SO(2n'_- + 1) \times \SO(2n''_- + 1) \times \GL(d_\rho)$ where $n'_- + n''_- = \frac{1}{2} \dim W_-$, so that we get $\mathbf{M}^! \in \Endo_{\elli}(\tilde{M})$ (see Remark \ref{rem:Endo-M});
	\item the factor $\GL(d_\rho)$ in $M^!$ goes into $\SO(2n'+1)$ (resp.\ $\SO(2n''+1)$) when $(\rho, a) \in \mathrm{Jord}(\phi')$ (resp.\ $(\rho, a) \in \mathrm{Jord}(\phi'')$).
\end{itemize}

Let $\mathbf{G}^!_-$ be $\mathbf{M}^!$ with $\mathrm{GL}(d_\rho)$ removed, thus $\mathbf{G}^!_- \in \Endo_{\elli}(\tilde{G}_-)$. Definition \ref{def:r-rho} gives $r^{G^!}_{\rho, x}: \mathrm{Groth}(G^!) \to \mathrm{Groth}(G^!_-)$ relative to $M^! \subset G^!$.

\begin{proposition}[{\cite[Corollary 3.3]{C24}}]\label{prop:alpha}
	Let $\mathbf{G}^!$, $\phi$, $\phi^!$ and $(\rho, a)$ be as above. For $x := \frac{a-1}{2}$, we have the following identity in $\mathrm{Groth}(\tilde{G}_-)$:
	\[ \alpha r^{\tilde{G}}_{\rho, x} \trans_{\mathbf{G}^!, \tilde{G}} (S\Theta^{G^!}_{\phi^!}) = \trans_{\mathbf{G}^!_-, \tilde{G}_-} r^{G^!}_{\rho, x}(S\Theta^{G^!}_{\phi^!}), \]
	where
	\[ \alpha := \begin{cases}
		1, & (\rho, a) \in \mathrm{Jord}(\phi') \\
		\omega_\rho(-1), & (\rho, a) \in \mathrm{Jord}(\phi'').
	\end{cases}\]
\end{proposition}

Proposition \ref{prop:alpha} is a metaplectic analog of Moeglin's result \cite[\S 2.6]{Moe14}, or more precisely of \cite[(6.3)]{Xu17}. In the diagram of the latter reference, there are a priori two summands on the right hand side, but only the one corresponding to $M^!$ contributes. The sign $\alpha$ is a metaplectic feature here. It arises when commuting parabolic induction and transfer \cite[Theorem 3.4.6]{Li19}.

\subsection{Reduction to smaller metaplectic groups}\label{sec:reduction-smaller}
Return to the setting of Theorem \ref{prop:main}. We focus on the following scenario:
\begin{itemize}
	\item $\pi \in \Pi_{2, -}(\tilde{G})$ belonging to $\mathcal{G}_{\bpsi}^{\pm}$;
	\item $\sigma := \mathrm{TW}^*(\pi) \in \Pi_2(G^{\pm})$, which belongs to $\mathcal{G}^\pm$;
	\item $\phi \in \Phi_{2, \mathrm{bdd}}(\tilde{G})$ is the common L-parameter of $\pi$ and $\sigma$ by Lemma \ref{prop:phi-tempered}.
\end{itemize}

The discussions below are largely modeled on \cite[Chapter 8]{Rel18}. As in \S\ref{sec:Jacquet-transfer}, given $(\rho, a) \in \mathrm{Jord}(\phi)$, we put
\[ \tilde{G}_- := \widetilde{\mathrm{Sp}}(W_-), \quad W_- \subset W, \quad \dim W_- = \dim W - 2d_\rho. \]

The story for $G^{\pm}$ is parallel. We put $G^{\pm}_- := \SO(V^{\pm}_-)$ where $V^{\pm}_- \subset V^{\pm}$ has discriminant $1$, Hasse invariant $\pm 1$ and satisfies $\dim V^{\pm}_- = \dim V^{\pm} - 2d_\rho$. We view $G^{\pm}_- \times \GL(d_\rho)$ as the Levi factor of a parabolic subgroup $P^{\pm} \supset P_{\min}^{\pm}$ of $G^{\pm}$.

\begin{lemma}\label{prop:Jord-r}
	Suppose $n > 1$, then there exists $(\rho, a) \in \mathrm{Jord}(\phi)$ such that
	\begin{enumerate}[(i)]
		\item $\rho$ is unramified, $a \geq 2$ and
		\[ r^{\tilde{G}}_{\rho, (a-1)/2}(\pi) \neq 0, \quad r^{G^{\pm}}_{\rho, (a-1)/2}(\sigma) \neq 0; \]
		\item there exists an irreducible representation $\pi^\flat$ of $\tilde{G}_-$ (resp.\ $\sigma^\flat$ of $G^{\pm}_-(F)$) such that $\pi \hookrightarrow i_{\tilde{P}}\left(\rho |\cdot|^{(a-1)/2} \boxtimes \pi^\flat \right)$ (resp.\ $\sigma \hookrightarrow i_{P^{\pm}}\left( \rho |\cdot|^{(a-1)/2} \boxtimes \sigma^\flat \right)$) for an appropriate parabolic subgroup $P$ (resp.\ $P^{\pm}$); moreover we have $P \leftrightarrow P^{\pm}$.
	\end{enumerate}
\end{lemma}
\begin{proof}
	By leveraging Theorem \ref{prop:compatibility-1}, the case of $\pi$ reduces to that of $\sigma$ since they share the same $L$-parameter $\phi$.
	
	By the description of $\mathcal{G}^{\pm}$ in terms of cuspidal supports and the assumption $n > 1$, there exists $(\rho, x)$ such that $\rho$ is an unramified unitary character of $F^{\times}$, $x \in \CC$, and $r_{P^{\pm}}(\sigma)$ has an irreducible quotient of the form $\rho|\cdot|^x \boxtimes \sigma^\flat$ for appropriate $P^{\pm}$. In particular, $r^{G^{\pm}}_{\rho, x}(\sigma) \neq 0$. By adjunction, $\sigma$ embeds into $i_{P^{\pm}}\left( \rho|\cdot|^x \boxtimes \sigma^\flat \right)$. 
	
	In \cite[\S 8.3.4]{Rel18} it is shown that there exists $a \geq 2$ such that $x = \frac{a-1}{2}$ and $(\rho, a) \in \mathrm{Jord}(\phi)$.
\end{proof}

For all $(\rho, a) \in \mathrm{Jord}(\phi)$ with $\rho$ unramified (hence $d_\rho = 1$) and $a \geq 2$, define
\begin{equation}
	\phi_- \in \Phi_{\mathrm{bdd}}(\tilde{G}_-)
\end{equation}
by the following recipe: if $a > 2$, we replace $\rho \boxtimes r(a)$ by $\rho \boxtimes r(a-2)$ in the decomposition of $\phi$ to obtain $\phi_-$; if $a=2$ then $\phi_-$ is simply obtained by removing $\rho \boxtimes r(2)$ from $\phi$. It is the empty parameter if $n=1$. Observe that $\phi_-$ is of good parity, but not necessarily in $\Phi_{2, \mathrm{bdd}}(\tilde{G}_-)$.

Write $\phi = \bigoplus_{i \in I} \phi_i$. The description in \eqref{eqn:S} reduces to $S_\phi = \prod_{i \in I} \Or(1, \CC)$. Given $(\rho, a)$ as above, take the unique $i_0 \in I$ such that $\phi_{i_0} = \rho \boxtimes r(a)$. There is a canonical homomorphism $S_\phi \to S_{\phi_-}$ described as follows.

\begin{enumerate}
	\item If $a > 2$ and $\rho \boxtimes r(a-2) \notin \mathrm{Jord}(\phi)$, then $S_{\phi_-} = \prod_{i \in I} \Or(1, \CC)$ and $S_\phi \rightiso S_{\phi_-}$ is tautological.
	\item If $a > 2$ and $\rho \boxtimes r(a-2) = \phi_{i_1}$ for some $i_1 \in I$, then
	\[ S_{\phi_-} = \Or(2, \CC) \times \prod_{i \in I \smallsetminus \{i_0, i_1 \}} \Or(1, \CC), \]
	and $S_\phi \to S_{\phi_-}$ is $\Or(1, \CC) \times \Or(1, \CC) \hookrightarrow \Or(2, \CC)$ on the indices $i_0, i_1$, identity elsewhere.
	\item If $a = 2$ then $S_{\phi_-} = \prod_{i \in I \smallsetminus \{i_0\}} \Or(1, \CC)$, and $S_\phi \to S_{\phi_-}$ is the evident projection.
\end{enumerate}

One deduces a canonical homomorphism $\EuScript{S}_\phi \to \EuScript{S}_{\phi_-}$. The following is evident.

\begin{lemma}\label{prop:S-minus}
	There is a canonical short exact sequence
	\[ 1 \to \EuScript{T} \to \EuScript{S}_\phi \to \EuScript{S}_{\phi_-} \to 1. \]
	\begin{itemize}
		\item In the case 1 above, $\EuScript{T} = \{1\}$.
		\item In the case 2 above, $\EuScript{T} \simeq \bmu_2$ embedded diagonally in the $\{i_0, i_1\}$-summands of $\EuScript{S}_\phi$.
		\item In the case 3 above, $\EuScript{T} \simeq \bmu_2$ embedded as the $i_0$-summand of $\EuScript{S}_\phi$.
	\end{itemize}
\end{lemma}

\begin{definition}
	Given $\chi \in \EuScript{S}_\phi^\vee$ that is trivial on $\EuScript{T}$, we denote by $\chi_-$ the resulting character of $\EuScript{S}_{\phi_-}$.
\end{definition}

One can also view $\phi_-$ as an L-parameter of $G^{\pm}_- = \SO(V^\pm_-)$.

For each $s \in S_\phi$, we have $s^2 = 1$, so the data $\mathbf{G}^! \in \Endo_{\elli}(\tilde{G})$ and $\phi^! \in \Phi_{2, \mathrm{bdd}}(G^!)$ such that $\phi^! \mapsto \phi$ via \eqref{eqn:Phi-Endo} can be constructed as in \eqref{eqn:T}.

Given $(\rho, a) \in \mathrm{Jord}(\phi)$ as in Lemma \ref{prop:Jord-r}, we obtain $\tilde{G}_-$ and $\phi_-$; the recipe in \S\ref{sec:Jacquet-transfer} also furnishes $\mathbf{G}^!_- \in \Endo_{\elli}(\tilde{G}_-)$. Replacement of $\rho \boxtimes r(a)$ in $\phi^!$ by $\rho \boxtimes r(a-2)$ (resp.\ removal) when $a > 2$ (resp.\ $a=2$) in the corresponding $\SO$-factor of $G^!$ yields $\phi^!_- \in \Phi_{\mathrm{bdd}}(G^!_-)$, which is of good parity.

\begin{lemma}\label{prop:r-transfer-prep}
	Suppose $n > 1$. Let $(\rho, a) \in \mathrm{Jord}(\phi)$ be as in Lemma \ref{prop:Jord-r}. The datum $(\mathbf{G}^!_-, \phi^!_-)$ is the one constructed from $\phi_- \in \Phi_{\mathrm{bdd}}(\tilde{G}_-)$ and the image $s_- \in S_{\phi_-}$ of $s \in S_\phi$, by the construction surrounding \eqref{eqn:T}.
\end{lemma}
\begin{proof}
	Clear by unraveling the constructions, especially the description of $S_\phi \to S_{\phi_-}$.
\end{proof}

\begin{lemma}\label{prop:r-transfer-aux}
	Suppose $n > 1$. Let $(\rho, a) \in \mathrm{Jord}(\phi)$ be as in Lemma \ref{prop:Jord-r} and $s \in S_\phi$. Define $(\mathbf{G}^!_-, \phi_-)$ by the foregoing constructions. Then $\trans_{\mathbf{G}^!_-, \tilde{G}_-} (S\Theta^{G^!_-}_{\phi^!_-})$ depends only on the image $x_-$ of $s$ in $\EuScript{S}_{\phi_-}$.
\end{lemma}
\begin{proof}
	Recall that the kernel $\EuScript{T}$ of $S_\phi \twoheadrightarrow \EuScript{S}_{\phi_-}$ is non-trivial only when (i) $a=2$, or (ii) $a > 2$ and $(\rho, a-2) \in \mathrm{Jord}(\phi)$.

	In the case (i), $\phi_- \in \Phi_{2, \mathrm{bdd}}(\tilde{G}_-)$ and $\EuScript{T}$ is the summand $\bmu_2$ indexed by $(\rho, 2)$. Translation of $s$ by $\EuScript{T}$ has no effect on $(\mathbf{G}^!_-, \phi^!_-)$ since $\EuScript{T} = \Ker[S_\phi \to S_{\phi_-}]$, in view of Lemma \ref{prop:r-transfer-prep}.
		
	Consider the case (ii), $\phi_- \notin \Phi_{2, \mathrm{bdd}}(\tilde{G}_-)$ since it contains $2 \rho \boxtimes r(a-2)$, and $\EuScript{T}$ is the diagonal of the $\bmu_2 \times \bmu_2$ indexed by $(\rho, a)$ and $(\rho, a-2)$. Denote by $s'$ the translation of $s$ by $(-1, -1) \in \EuScript{T}$.
	\begin{itemize}
		\item Suppose the $\bmu_2 \times \bmu_2$-part of $s$ is $(1, -1)$. Then $s$ and $s'$ give rise to the same $(\mathbf{G}^!_-, \phi^!_-)$: each $\SO$-factor of $G^!_-$ contributes a $(\rho, a-2)$ to $\mathrm{Jord}(\phi_-)$.
		
		\item Suppose the $\bmu_2 \times \bmu_2$-part from $s$ is $(1, 1)$. Removing $\rho \boxtimes r(a)$ and $\rho \boxtimes r(a-2)$ from $\phi$ yields an L-parameter $\underline{\phi}$ for some $\tilde{\underline{G}} = \Mp(\underline{W})$. Removing the $\bmu_2 \times \bmu_2$-part of $s$ or $s'$ yields the same $\underline{s} \in S_{\underline{\phi}}$. To $(\underline{\phi}, \underline{s})$ is attached $\mathbf{\underline{G}}^! \in \Endo_{\elli}(\tilde{\underline{G}})$ and $\underline{\phi}^! \in \Phi_{2, \mathrm{bdd}}(\underline{G}^!)$ such that $\underline{\phi}^! \mapsto \underline{\phi}$ via \eqref{eqn:Phi-Endo} for $\tilde{\underline{G}}$.
			
		Note that $\underline{M} := \underline{G} \times \GL(a-2)$ is a Levi subgroup of $G_-$. We deduce the datum $\mathbf{\underline{M}}^! \in \Endo_{\elli}(\tilde{\underline{M}})$ from $\mathbf{\underline{G}}^!$, together with the L-parameter $\underline{\phi}^! \oplus (2 \rho \boxtimes r(a-2))$ for $\underline{M}^!$. There are exactly two ways to lift $(\mathbf{\underline{M}}^!, \underline{\phi}^!)$ to the level $\tilde{G}_-$, namely by embedding $\GL(a-2)$ into $\SO(2n'_- + 1)$ or $\SO(2n''_- + 1)$ where $(n'_-, n''_-)$ describes the desired element of $\Endo_{\elli}(\tilde{G}_-)$. These two liftings correspond to the $(\mathbf{G}^!_-, \phi^!_-)$ attached to $s$ and $s'$, respectively.
			
		Let $\tau$ be the irreducible representation of $\GL(a-2, F)$ parameterized by $\rho \boxtimes r(a-2)$. Due to the appearance of $-1$ when commuting induction and transfer (this is a metaplectic feature --- see \cite[Theorem 3.4.6]{Li19}), the $\trans_{\mathbf{G}^!_-, \tilde{G}_-}(S\Theta^{G^!_-}_{\phi^!_-})$ attached to $s$ and $s'$ differ by $\omega_\tau(-1)$. Since $\phi$ is of good parity, $\rho \boxtimes r(a-2)$ is self-dual of symplectic type and we get $\omega_\tau(-1) = 1$.
	\end{itemize}
	
	In each situation, we see $\trans_{\mathbf{G}^!_-, \tilde{G}_-} (S\Theta^{G^!_-}_{\phi^!_-})$ depends only on $x_-$, as desired.
\end{proof}

\subsection{Completion of the proof}\label{sec:completion-proof}
Retain the assumptions in \S\ref{sec:reduction-smaller}. In view of Lemma \ref{prop:phi-tempered}, we may write $\pi = \pi_{\phi, \chi}$ and $\sigma = \sigma_{\phi, \chi^\circ}$ for some $\chi, \chi^\circ \in \EuScript{S}_\phi^\vee$.

\begin{lemma}\label{prop:r-transfer}
	Suppose $n > 1$. Let $(\rho, a) \in \mathrm{Jord}(\phi)$ be as in Lemma \ref{prop:Jord-r}, then:
	\begin{enumerate}[(i)]
		\item $\chi^\circ$ is trivial on $\EuScript{T}$ and we have
		\[ r^{G^{\pm}}_{\rho, (a-1)/2} \left( \sigma_{\phi, \chi^\circ} \right) = \sigma_{\phi_-, \chi^\circ_-} \]
		in $\mathrm{Groth}(G^{\pm}_-)$;
		\item $\chi \nu_\phi$ is trivial on $\EuScript{T}$ and we have
		\[ r^{\tilde{G}}_{\rho, (a-1)/2} \left( \pi_{\phi, \chi} \right) = \pi_{\phi_-, (\chi\nu_\phi)_- \nu_{\phi_-}} \]
		in $\mathrm{Groth}(\tilde{G}_-)$.
	\end{enumerate}
\end{lemma}
\begin{proof}
	The assertion (i) is contained in the proof of \cite[\S 8.3.4]{Rel18}. Below we sketch a variant of \textit{loc.\ cit.}\ for (ii).
	
	After a Fourier inversion on $\EuScript{S}_\phi = S_\phi$, Theorem \ref{prop:Luo-endo} reads
	\begin{equation}\label{eqn:r-transfer-0}
		\Tr\left( \pi_{\phi, \chi}\right) = |S_\phi|^{-1} \sum_{s \in S_\phi} (\chi \nu_\phi)(s) \trans_{\mathbf{G}^!, \tilde{G}} (S\Theta^{G^!}_{\phi^!}),
	\end{equation}
	where $(\mathbf{G}^!, \phi^!)$ is determined from $(\phi, s)$. Note that $\phi^! = (\phi', \phi'') \in \Phi_{2, \mathrm{bdd}}(G^!)$.
	
	Apply $r^{\tilde{G}}_{\rho, (a-1)/2}$ to \eqref{eqn:r-transfer-0}. By the Proposition \ref{prop:alpha}, the right-hand side becomes
	\begin{equation*}
		|S_\phi|^{-1} \sum_{s \in S_\phi} (\chi \nu_\phi)(s) \alpha(s) \trans_{\mathbf{G}^!_-, \tilde{G}_-} r^{G^!}_{\rho, \frac{a-1}{2}}(S\Theta^{G^!}_{\phi^!}),
	\end{equation*}
	where
	\begin{equation*}
		\alpha(s) := \begin{cases}
			1, & (\rho, a) \in \mathrm{Jord}(\phi'), \\
			\omega_\rho(-1), & (\rho, a) \in \mathrm{Jord}(\phi'').
		\end{cases}
	\end{equation*}
	In general, $\alpha \in S_\phi^\vee$. Since $\rho$ is unramified here, $\alpha = \mathbf{1}$.
	
	Construct $(\mathbf{G}^!_-, \phi^!_-)$ from $(\rho, a) \in \mathrm{Jord}(\phi)$ and $s$; see Lemma \ref{prop:r-transfer-prep}. Applying Moeglin's results \cite[below (8.3)]{Rel18} for split odd orthogonal groups, we get
	\[ r^{\tilde{G}}_{\rho, \frac{a-1}{2}} (\pi_{\phi, \chi}) = |S_\phi|^{-1} \sum_{s \in S_\phi} (\chi\nu_\phi)(s) \trans_{\mathbf{G}^!_-, \tilde{G}_-} ( S\Theta^{G^!_-}_{\phi^!_-} ). \]

	For every $s \in S_\phi$, denote by $s_-$ (resp.\ $x_-$) its image in $S_{\phi_-}$ (resp.\ $\EuScript{S}_{\phi_-}$). Inside the sum, $(\mathbf{G}^!_-, \phi^!_-)$ is determined by $(\phi_-, s_-)$ (Lemma \ref{prop:r-transfer-prep}). However, Lemma \ref{prop:r-transfer-aux} implies $\trans_{\mathbf{G}^!_-, \tilde{G}_-} ( S\Theta^{G^!_-}_{\phi^!_-} )$ depends only on $(\phi_-, x_-)$. Since $r^{\tilde{G}}_{\rho, (a-1)/2} (\pi_{\phi, \chi}) \neq 0$ (Lemma \ref{prop:Jord-r}), the sum is nonzero. It follows that $\chi \nu_\phi$ must be trivial on $\EuScript{T}$. Now \eqref{eqn:r-transfer-0} can be rewritten as
	\begin{equation*}
		r_{\rho, \frac{a-1}{2}} (\pi_{\phi, \chi}) = |\EuScript{S}_{\phi_-}|^{-1} \sum_{x_- \in \EuScript{S}_{\phi_-}} (\chi \nu_\phi)_-(x_-) \trans_{\mathbf{G}^!_-, \tilde{G}_-} (S\Theta^{G^!_-}_{\phi^!_-}).
	\end{equation*}
	The $(\mathbf{G}^!_-, \phi^!_-)$ in the sum above is defined by picking a preimage $s_- \in S_{\phi_-}$ with $s_-^2 = 1$ for each $x_-$. By \eqref{eqn:T} and Theorem \ref{prop:Luo-endo}, one may write $T_{\phi_-, s_-} = \epsilon\left( \phi_-^{s_- = -1} \right) \trans_{\mathbf{G}^!_-, \tilde{G}_-} (S\Theta^{G^!_-}_{\phi^!_-})$ as $T_{\phi_-, x_-}$ since it depends only on $(\phi_-, x_-)$.
	
	On the other hand, as $\phi_-$ is of good parity, one can check that $\epsilon\left( \phi_-^{s_- = -1} \right) = \nu_{\phi_-}(x_-)$; see \cite[Proposition 4.3.4]{Li24a} for details. Hence by applying Theorem \ref{prop:Luo-endo} to $\tilde{G}_-$ and $\phi_-$, we reach
	\begin{align*}
		r_{\rho, \frac{a-1}{2}} (\pi_{\phi, \chi^\circ}) & = |\EuScript{S}_{\phi_-}|^{-1} \sum_{x_- \in \EuScript{S}_{\phi_-}} (\chi\nu_\phi)_-(x_-) \nu_{\phi_-}(x_-) T_{\phi_-, x_-} \\
		& = \Tr \left( \pi_{\phi_-, (\chi \nu_\phi)_- \nu_{\phi_-}} \right)
	\end{align*}
	by another Fourier inversion. This is the desired equality in $\mathrm{Groth}(\tilde{G}_-)$.
\end{proof}

The proofs of the main Theorem \ref{prop:main} and Proposition \ref{prop:phi-nr} can now be completed as follows.

\begin{lemma}\label{prop:to-Jacquet}
	If the statement $(\phi^\circ, \chi^\circ) = (\phi, \chi\nu_\phi)$ in Theorem \ref{prop:main} holds for all tempered irreducible genuine representations of metaplectic groups of rank $< n$, then it holds for all irreducible genuine square-integrable representations $\pi$ of $\tilde{G}$. Here the representations are assumed to belong to $\mathcal{G}_{\bpsi}^{\pm}$ or its avatars.
\end{lemma}
\begin{proof}
	In view of Lemma \ref{prop:basic-case}, we may assume $n > 1$. Choose $(\rho, a) \in \mathrm{Jord}(\phi)$ as in Lemma \ref{prop:Jord-r}. Both $r^{G^{\pm}}_{\rho, (a-1)/2}( \sigma_{\phi, \chi^\circ})$ and $r^{\tilde{G}}_{\rho, (a-1)/2}( \pi_{\phi, \chi})$ are irreducible by Lemma \ref{prop:r-transfer}. On the other hand, they belong to the avatars of $\mathcal{G}^{\pm}$ and $\mathcal{G}^{\pm}_\bpsi$ for groups of the same type with smaller ranks, respectively. Denote by $\mathrm{TW}_-$ the avatar of $\mathrm{TW}$ for $\tilde{G}_-$. By Theorem \ref{prop:compatibility-1},
	\[ r^{G^{\pm}}_{\rho, (a-1)/2}( \sigma_{\phi, \chi^\circ}) = \mathrm{TW}_-^* \left( r^{\tilde{G}}_{\rho, (a-1)/2}( \pi_{\phi, \chi})\right) \]
	in the Iwahori--spherical part of $\mathrm{Groth}(G^{\pm}_-)$. Hence Lemma \ref{prop:r-transfer} implies
	\[ \sigma_{\phi_-, \chi^\circ_-} \simeq \mathrm{TW}_-^* \left( \pi_{\phi_-, (\chi\nu_\phi)_- \nu_{\phi_-}} \right). \]
	
	By assumption, we get
	\[ (\chi \nu_\phi)_- \nu_{\phi_-} = \chi^\circ_- \nu_{\phi_-}. \]
	This implies $(\chi \nu_\phi)_- = \chi^\circ_-$, which in turn gives $\chi \nu_\phi = \chi^\circ$ as desired.
\end{proof}

\begin{proof}[Proof of Proposition \ref{prop:phi-nr}]
	Let $\phi \in \Phi(\tilde{G})$ be the L-parameter of $\pi$. The assertion is equivalent to that $\rho$ is unramified for every $(\rho, a) \in \mathrm{Jord}(\phi)$. Upon recalling the proofs of Lemmas \ref{prop:phi-tempered} and \ref{prop:to-tempered}, it suffices to treat square-integrable $\pi$.
	
	For the basic case $n = 1$, we have seen in Remark \ref{rem:basic-L} that $(\rho, a) = (\xi, 2)$ for some unramified quadratic character $\xi$ of $F^{\times}$, in both the $+$ and $-$ cases.
	
	Now suppose $n > 1$. In Lemma \ref{prop:r-transfer} (ii) we picked $(\rho, a) \in \mathrm{Jord}(\phi)$ with $\rho$ unramified, $a \geq 2$ , and produced $\phi_- \in \Phi_{\mathrm{bdd}}(\tilde{G}_-)$ such that
	\begin{itemize}
		\item $\tilde{G}_-$ is a metaplectic group of rank $n-1$;
		\item the multi-set $\mathrm{Jord}(\phi_-)$ is obtained from $\mathrm{Jord}(\phi)$ by either removing $(\rho, 2)$ (when $a = 2$) or replacing $(\rho, a)$ by $(\rho, a-2)$ (when $a > 2$);
		\item if we put $\eta := (\chi\nu_\phi)_- \nu_{\phi_-} \in \EuScript{S}_{\phi_-}^\vee$, then $\pi_{\phi_-, \eta}$ lies in the avatar of $\mathcal{G}_{\bpsi}^{\pm}$ for $\tilde{G}_-$.
	\end{itemize}
	
	It follows from induction that $\rho'$ is unramified for all $(\rho', a') \in \mathrm{Jord}(\phi_-)$, hence the same holds for $\mathrm{Jord}(\phi)$.
\end{proof}

The argument above can be compared with Xu's algorithm for cuspidal supports, see \cite[Proposition 8.1]{Xu17}.

\section*{Declarations}
\paragraph*{Conflict of Interest} The authors declare no conflict of interest.

\printbibliography[heading=bibintoc]

\vspace{1em}
\begin{flushleft} \small
	F. Chen: Yau Mathematical Sciences Center, Tsinghua University, 100084 Beijing, People's Republic of China. \\
	E-mail address: \href{mailto:fchen@mail.tsinghua.edu.cn}{\texttt{fchen@mail.tsinghua.edu.cn}}
\end{flushleft}

\begin{flushleft} \small
	W.-W. Li: Beijing International Center for Mathematical Research / School of Mathematical Sciences, Peking University. No.\ 5 Yiheyuan Road, 100871 Beijing, People's Republic of China. \\
	E-mail address: \href{mailto:wwli@bicmr.pku.edu.cn}{\texttt{wwli@bicmr.pku.edu.cn}}
\end{flushleft}

\end{document}